\documentclass[a4paper]{article}
\usepackage[infoshow]{tracefnt}
\usepackage[inner=3.5cm,outer=2.5cm]{geometry}
\usepackage{graphicx}
\usepackage{minitoc}
\usepackage{float,caption}
\usepackage{amsmath,amssymb,mathrsfs,amsthm,bbold,amsfonts}
\usepackage{pdfpages}
\usepackage[square,numbers,sectionbib]{natbib}
\usepackage{chapterbib}
\usepackage{bibunits}
\usepackage{array}
\usepackage{subfig}
\usepackage[nottoc, notlof, notlot]{tocbibind}
\newtheorem{proposition}{Proposition}[section]
\newtheorem{lemma}[proposition]{Lemma}

\newtheorem{theorem}[proposition]{Theorem}
\newtheorem{corollary}[proposition]{Corollary}
\newtheorem{definition}[proposition]{Definition}
\newtheorem{remark}[proposition]{Remark}

\usepackage{url} 
\numberwithin{equation}{section}
\usepackage{setspace}
\newcolumntype{M}[1]{>{\raggedright}m{#1}}
\begin{document}
\title{Stability estimates for a Robin coefficient in the two-dimensional Stokes system \footnote{This work was partially funded by the ANR-08-JCJC-013-01 (M3RS) project 
headed by C. Grandmont and the ANR-BLAN-0213-02 (CISIFS) project headed by L. Rosier.}}
\author{Muriel Boulakia \thanks{Universit\'e Pierre et Marie Curie-Paris 6, UMR 7598, Laboratoire Jacques-Louis Lions, 75005 Paris, France \& INRIA, Projet REO, Rocquencourt, BP 105, 78153 Le Chesnay cedex, France} \and Anne-Claire Egloffe  \footnotemark[3]  \and C\'eline Grandmont  \thanks{ INRIA, Projet REO, Rocquencourt, BP 105, 78153 Le Chesnay cedex, France \& Universit\'e Pierre et Marie Curie-Paris 6, UMR 7598, Laboratoire Jacques-Louis Lions, 75005 Paris, France.} }
\maketitle
\begin{abstract}
In this paper, we consider the Stokes equations and we are concerned with the inverse problem of identifying a Robin coefficient on some non accessible part of the boundary from available data on the other part of the boundary. We first study the identifiability of the Robin coefficient and then we establish a stability estimate of logarithm type thanks to a Carleman inequality due to A. L. Bukhgeim~\cite{bukhgeim} and under the assumption that the velocity of a given reference solution stays far from $0$ on a part of the boundary where Robin conditions are prescribed.
\end{abstract}
\textit{Keywords:} Inverse boundary coefficient problem, Stokes system, Robin boundary condition, Identifiability, Carleman inequality, Logarithmic stability estimate.
%
%
%
\section{Introduction}
Let us consider an open Lipschitz bounded connected domain $\Omega$ of $\mathbb{R}^d$, $d \geq 2$. We assume that the boundary $\partial \Omega$ is composed of two open non-empty parts  $\Gamma_0$  and  $\Gamma_e$ such that $\Gamma_e \cup \Gamma_0 =\partial \Omega$  and  $\overline{\Gamma}_e \cap \overline{\Gamma}_0 = \emptyset$ (Figure~\ref{ouvert} gives an example of such a geometry in dimension 2). 
\par
We denote by $n$ the exterior unit normal to $\Omega$ and let $\displaystyle \tau=(\tau_1, \hdots,\tau_{d-1})$ be $d-1$ vectors of $\mathbb{R}^d$ such that $(n,\tau)$ is an orthogonal basis of $\mathbb{R}^d$.
\begin{figure}[htp]
\begin{center}
\includegraphics[scale=0.17,clip]{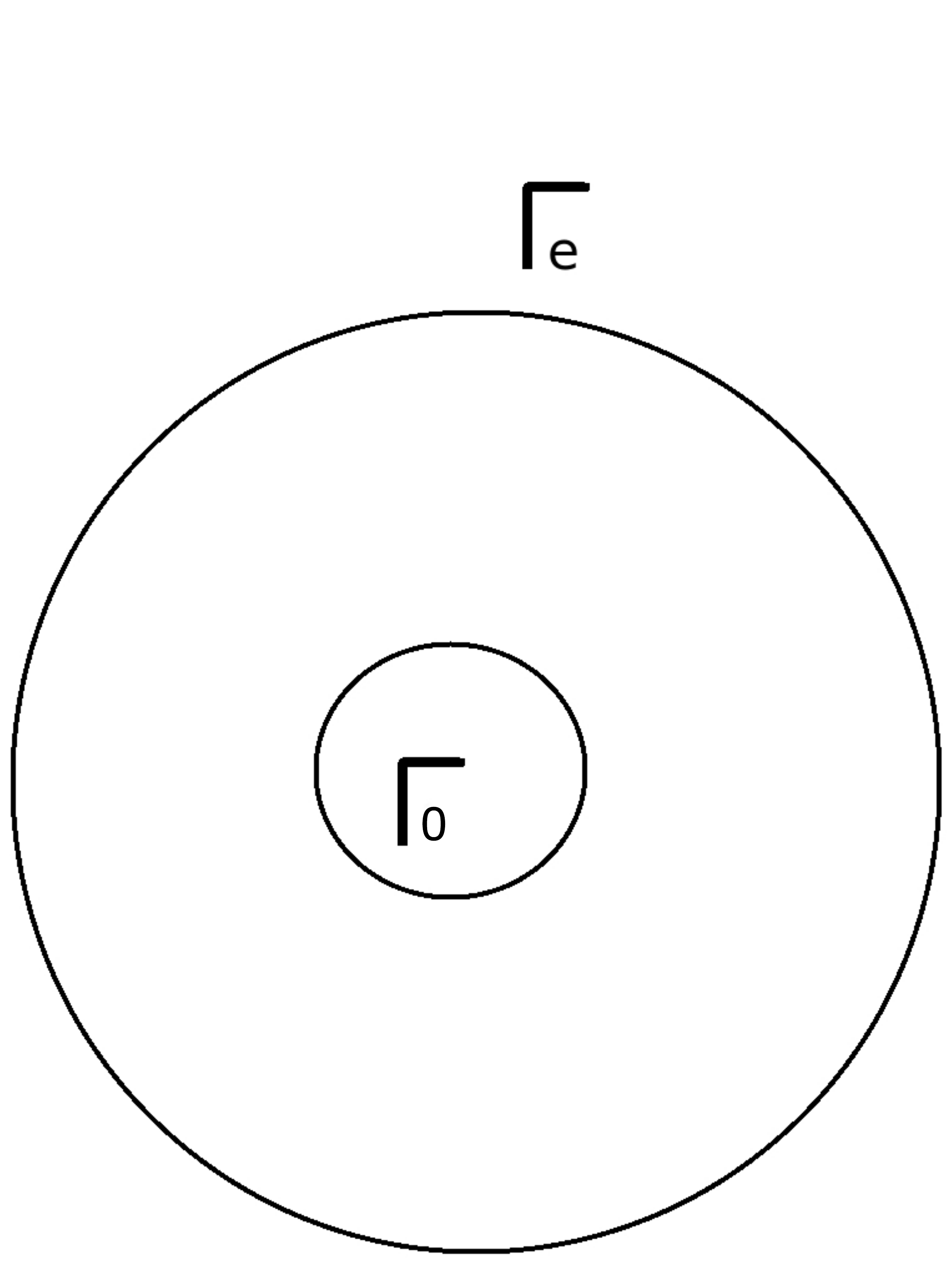}
\caption{Example of an open set $\Omega$ such that $\displaystyle\Gamma_e \cup \Gamma_0 =\partial \Omega$  and  $\displaystyle\overline{\Gamma}_e \cap \overline{\Gamma}_0 = \emptyset$ in dimension 2.}
\label{ouvert}
\end{center}
\end{figure}
\par
We introduce the following boundary problem:
\begin{equation} 
	\left \{   
		\begin{array}{cccl}
                 \displaystyle      \partial_t u(t,x) - \Delta u(t,x) + \nabla p(t,x)			&	=	&	0,							&\forall x \in \Omega,  \forall\, t > 0, \\ 
              \displaystyle                            div \textrm{ } u(t,x) 					&	=	&	0,							& \forall x \in \Omega,  \forall\, t > 0, \\  
             \displaystyle                   \frac{\partial u} {\partial n}(t,x) - p(t,x)n(x)			&	=	&	g(t,x), 						& \forall x \in \Gamma_e,  \forall\, t > 0, \\   
             \displaystyle                     \frac{\partial u} {\partial n}(t,x) - p(t,x)n(x)	+ q(x)u(t,x)&	=	&	0,							& \forall x \in \Gamma_0,  \forall\, t > 0, \\
             \displaystyle                  u(0,x)					&	=	&	u_0(x),									& \forall x \in \Omega.
                   \end{array}     
         \right.
         \label{eq1} 
\end{equation}
Notice that we assume that the Robin coefficient $q$ defined on $\Gamma_0$ only depends on the space variable. Our objective is to determine the coefficient $q$ from the values of $u$ and $p$ on $\Gamma_e$.
\par
Such kinds of systems naturally appear in the modeling of biological problems like, for example, blood flow in the cardiovascular system (see~\cite{quarteroni} and~\cite{irene}) or airflow in the lungs  (see~\cite{Baffico}). For an introduction on the modeling of the airflow in the lungs and on different boundary conditions which may be prescribed, we refer to \cite{egloffe_these}. The part of the boundary $\Gamma_e$ represents a physical boundary on which measurements are available and $\Gamma_0$  represents an artificial boundary on which Robin boundary conditions or mixed boundary conditions involving the fluid stress tensor and its flux at the outlet are prescribed. 
\par
Similar inverse problems have been widely studied for the Laplace equation \cite{alessandrini_rondi}, \cite{ccb}, \cite{chaabane_leblond}, \cite{chaabane_jaoua}, \cite{cheng-choulli-lin} and \cite{sincich}. This kind of problems arises in general in corrosion detection which consists in determining a Robin coefficient on the inaccessible portion of the boundary thanks to electrostatic measurements performed on the accessible boundary. Most of these papers prove a logarithmic stability estimate (\cite{alessandrini_rondi},~\cite{ccb},~\cite{chaabane_leblond} and~\cite{cheng-choulli-lin}). We mention that, in~\cite{chaabane_jaoua}, S. Chaabane and M. Jaoua obtained both local and monotone global Lipschitz stability for regular Robin coefficient and under the assumption that the flux $g$ is non negative. Under the \textit{a priori} assumption that the Robin coefficient is piecewise constant, E. Sincich has obtained in~\cite{sincich} a Lipschitz stability estimate. To prove stability estimates, different approaches are developed in these papers. A first one consists in using the complex analytic function theory (see~\cite{alessandrini_rondi},~\cite{chaabane_leblond}). A characteristic of this method is that it is only valid in dimension 2. Another classical approach is based on Carleman estimates (see~\cite{ccb} and~\cite{cheng-choulli-lin}). In~\cite{ccb}, the authors use a result proved by K.D. Phung in~\cite{phung} to obtain a logarithmic stability estimate which is valid in any dimension for an open set $\Omega$ of class $\mathcal{C}^{\infty}$. This result has been generalized in~\cite{bourgeois} and~\cite{bourgeois_darde} to $\mathcal{C}^{1,1}$ and Lipschitz domains. Moreover, in~\cite{ccb}, the authors use semigroup theory to obtain a stability estimate in long time for the heat equation from the stability estimate for the Laplace equation. 
\par
In this article, we prove an identifiability result  and  a logarithmic stability estimate for the Stokes equations with Robin boundary conditions~\ref{eq1} under  the assumption that the velocity of a given reference solution stays far from $0$ on a part of the boundary where Robin conditions are prescribed. We would like to highlight why this assumption appears for the inverse problem of recovering a Robin coefficient. Let us consider $(u_i,p_i)$ be solutions of system~\ref{eq1} associated to $q=q_i$, for $i=1,2$. Using the boundary conditions on $\Gamma_0$, we obtain 
\begin{equation*}
(q_2-q_1)u_1= q_2(u_1-u_2) + \left(\frac{\partial u_1}{\partial n} -\frac{\partial u_2}{\partial n}\right )-(p_1-p_2)n.
\end{equation*}
When $u_1$ vanishes, difficulties occur to estimate the difference between the Robin coefficients $q_2-q_1$. In the case of the scalar Laplace equation, it is possible to determine the sign of $u$ solution of 
\begin{equation*} 
	\left \{   
		\begin{array}{cccl}
            \displaystyle          \Delta u		&	=	&	0,							&\text{ in }\Omega,  \\ 
             \displaystyle                   \frac{\partial u} {\partial n}			&	=	&	g,						& \text{ on } \Gamma_e, \\   
                \displaystyle                  \frac{\partial u} {\partial n}	+ qu&	=	&	0,							& \text{ on } \Gamma_0, \\
                   \end{array}     
         \right.
\end{equation*}
 on any compact subset $K \subset \Gamma_0$ under some positivity assumption on the flux $g$ (see~\cite{chaabane_jaoua}). Such a result comes from properties specific to harmonic functions, like for instance the maximum principle. When the flux $g$ has a variable sign, G. Alessandrini, L. Del Piero and L. Rondi provide in~\cite{alessandrini_rondi} a quantitative control of the vanishing rate of $u$ that allows to estimate the difference between the Robin coefficients on $\{x \in \Gamma_0 / d(x, \partial \Omega \backslash \Gamma_0) >d)\}$, for any $d>0$,  by using methods of complex analytic function theory. Moreover, G. Alessandrini and E. Sincich proved in~\cite{alessandrini_sincich} that the oscillation of $u$ on $\Gamma_0$ is bounded from below by a constant depending on the \textit{a priori} data only. To do so, they use unique continuation estimates for the Laplace equation.
 Due to the methods employed, it does not seem that we can extend these results to the Stokes system. This is why we estimate the Robin coefficient on a compact subset $K \subset \Gamma_0$ on which $u_1$ does not vanish. This estimate and the set $K$ depend on $u_1$, and knowing whether one can control our solution, for well chosen data, on the whole set $\Gamma_0$ or on any compact subset $K \subset \Gamma_0$  remains an open problem. Note however that in a really particular case (detailed in Remark~\ref{particular_case}), one can obtain a logarithmic estimate on the whole set $\Gamma_0$.
\par
The paper is organized as follows.
The second section contains preliminary results on the regularity of the solution. In the third section, we are interested in the identifiability of the Robin coefficient $q$. Under some regularity assumptions  and  using the theorem of unique continuation for the Stokes equations proved in~\cite{fabre}, we prove that if  two measurements of the velocity are equal on $(0,T) \times \Gamma$, where $\Gamma \subseteq \Gamma_e$ is a non-empty open subset of the boundary, then the two corresponding Robin coefficients are also equal on $\Gamma_0$.
Section~\ref{logarithmic} corresponds to the main part of our article. The results of this section are only valid in dimension 2. We prove a stability estimate, first for the stationary problem and then for the evolution problem. To do this, we use a global Carleman inequality due to  A. L. Bukhgeim which is only valid in dimension $2$ (see~\cite{bukhgeim}).  The stability estimate for the unsteady problem is deduced from the stability estimate for the stationary problem thanks to the semigroup theory. We end Section~\ref{logarithmic} by concluding remarks and perspectives to this work.
\par
When we are not more specific, $C$ is a generic constant, whose value may change and which only depends on the geometry of the open  set $\Omega$  and of  the boundaries $\Gamma_e$  and  $\Gamma_0$. Moreover, we denote indifferently by $|$ $|$ a norm on $\mathbb{R}^n$,  for any $n \geq 1$.
\par
We are going to start with some preliminary results which will be useful in the subsequent sections.
\section{Preliminary results}
%
%
In this section we study the well--posedness of the system and  the regularity of the solution.
\subsection{Regularity of the stationary problem}
\label{probleme_stat}
Let us first consider the stationary case: 
\begin{equation} \label{eq3}
	\left \{   
		\begin{array}{ccll}
                 \displaystyle      - \Delta u + \nabla p			&	=	&	f,							& \textrm{ in }  \Omega, \\ 
                    \displaystyle                      div \textrm{ } u 		&	=	&	0,							& \textrm{ in }  \Omega, \\  
                 \displaystyle              \frac{\partial u }{\partial n} - pn		&	=	&	g, 							& \textrm{ on  }  \Gamma_e, \\   
                 \displaystyle              \frac{\partial u }{\partial n} - pn	+ qu &	=	&	0,							& \textrm{ on  } \Gamma_0.  \\
                   \end{array}     
         \right. 
 \end{equation}
For $g \in H^{-\frac{1}{2}}(\Gamma_e)^d$ and $v \in H^{\frac{1}{2}}(\Gamma_e)^d$, we denote by $< g ,v>_{ -\frac{1}{2},\frac{1}{2},\Gamma_e } $ the image of $v$ by the linear form $g$.
\par
Let us  introduce some functional spaces: 
\begin{equation*}
V= \left\{ v \in H^1(\Omega)^d / div \textrm{ } v = 0 \textrm{ in } \Omega \right\}
\end{equation*}
 and 
\begin{equation*}
H= \overline{V}^{{L^2(\Omega)}^d}.
\end{equation*}
\begin{proposition} \label{existence1}
Let $\alpha > 0$, $f \in L^2(\Omega)^d$, $g \in H^{-\frac{1}{2}}(\Gamma_e)^d$ and $ q \in  L^{\infty}(\Gamma_0)$ be such that $q \geq \alpha$ on $\Gamma_0$. System~\ref{eq3} admits a unique solution $(u,p) \in V \times L^2(\Omega)$. Moreover, there exists a constant $C(\alpha)>0$ such that
\begin{equation} \label{eq4}
\|u\|_{H^1(\Omega)^d} \leq C(\alpha) (\|g\|_{ H^{-\frac{1}{2}}(\Gamma_e )^d} + \|f\|_{L^2(\Omega)^d}).
\end{equation}
\end{proposition}
\begin{proof}[Proof of Proposition~\ref{existence1}]
The variational formulation of the problem is: find $ u \in V $ such that for every $ v \in V $, 
\begin{equation*} 
\int_{\Omega } \nabla u: \nabla v + \int_{\Gamma_0} q u \cdot v = < g , v_{|\Gamma_e}>_{ -\frac{1}{2},\frac{1}{2},\Gamma_e }  + \int_{\Omega} f \cdot v.
\end{equation*}
For all $(u,v) \in V \times V$, we denote by
\begin{equation} \label{eq5}
 a_q(u,v)=\int_{\Omega } \nabla u: \nabla v + \int_{\Gamma_0} q u \cdot v ,
 \end{equation}
 and for all $v \in V$,
\begin{equation*}
L_1(v)=< g , v_{|\Gamma_e}>_{ -\frac{1}{2},\frac{1}{2},\Gamma_e }  +\int_{\Omega} f \cdot v.
\end{equation*}
We easily verify that $a_q$ is a continuous symmetric bilinear form. Since $q \geq \alpha > 0$, according to the generalized Poincar\'e inequality, the bilinear form $a_q$ is coercive on $V$. On the other hand, $ L_1 $ is a continuous linear form on $V$. Thus we prove the existence and uniqueness of $ u \in V $ solution of \ref{eq3} by using Lax-Milgram Theorem. We obtain simultaneously estimate $\ref{eq4}$. 
We prove the existence and uniqueness of $p \in L^2_0(\Omega)$ in a classical way, using De Rham Theorem. The fact that $p$ is unique in $L^2(\Omega)$ comes from the boundary conditions. We refer to~\cite{fabrie} for a complete proof in the case of Neumann boundary condition.
\end{proof}
Next we want to derive regularity properties of the solution.
Let us first recall existence and regularity results for the Stokes problem with Neumann boundary condition proved in~\cite{fabrie}.
\begin{proposition} \label{regularite_neumann}
Let  $k \in \mathbb{N}$. Assume that $\Omega$ is of class $\mathcal{C}^{k+1,1}$. We assume that: 
\begin{equation*}
(f,h) \in H^k(\Omega)^d \times H^{k+\frac{1}{2}}(\partial \Omega)^d.
\end{equation*}
Then the solution $(u,p)$ of
\begin{equation*} 
	\left \{   
		\begin{array}{ccll}
                \displaystyle        - \Delta u + \nabla p			&	=	&	f,							& \textrm{ in }  \Omega, \\ 
                  \displaystyle                        div \textrm{ } u 		&	=	&	0	,						& \textrm{ in }  \Omega, \\  
                  \displaystyle             \frac{\partial u }{\partial n} - pn		&	=	&	h 	,						& \textrm{ on  }  \partial \Omega, \\   
                   \end{array}     
         \right. 
\end{equation*}
belongs to $H^{k+2}(\Omega)^d \times H^{k+1}(\Omega)$ and there exists a constant $C>0$ such that: 
\begin{equation*}
\|u\|_{{H^{k+2}(\Omega)}^d}+ \|p\|_{H^{k+1}(\Omega)} \leq C(\|h\|_{H^{k+\frac{1}{2}}(\partial \Omega)^d } + \|f\|_{H^k(\Omega)^d}).
\end{equation*}
\end{proposition}
%
%
%
In order to study the Stokes system with Robin boundary conditions, one needs to specify to which space the Robin coefficient $q$ belongs. As stated in Proposition \ref{regstat}, we will assume that $q$ belongs to some Sobolev space $H^s(\Gamma_{0})$ where $s$ is large enough so that $qu_{|\Gamma_{0}}$ belongs to $H^r(\Gamma_{0})$ if $u_{|\Gamma_{0}}$ belongs to $H^r(\Gamma_{0})$. This stability in the Sobolev spaces will allow to apply the previous proposition (Proposition~\ref{regularite_neumann}). Before stating the regularity result, let us state the following lemma: 
\begin{lemma} \label{mpojsdhhbg}
Let $r, s \in\mathbb{R}$, with $s>\frac{d-1}{2}$ and $0\leq r \leq s$. Let $q \in H^s(\Gamma_0)$. The linear operator
\begin{equation*}
\begin{array}{rcl}
T: H^r(\Gamma_0) &\to &H^{r}(\Gamma_0) \\
u &\mapsto& qu
\end{array}
\end{equation*}
is continuous. Furthermore, the following estimate holds true 
\begin{equation*}
\|qu\|_{H^{r}(\Gamma_0)} \leq C \|q\|_{H^{s}(\Gamma_0)}\|u\|_{H^{r}(\Gamma_0)}.
\end{equation*}
\end{lemma}
\begin{proof}[Proof of Lemma~\ref{mpojsdhhbg}]
Since $s>\frac{d-1}{2}$, $H^s(\Gamma_0)$ is a Banach algebra (see~\cite{adams_fournier})  and thus $T \in \mathcal{L}(H^s(\Gamma_0),H^s(\Gamma_0))$ and $ \|T\|_s =\sup_{u \in H^s(\Gamma_0), u \not=0}\frac{ \|Tu\|_{H^s(\Gamma_0)} }{\|u\|_{H^s(\Gamma_0)}}\leq  \|q\|_{H^s(\Gamma_0)}  $. Moreover, since $H^s(\Gamma_0) \hookrightarrow L^{\infty}(\Gamma_0)$, $T \in \mathcal{L}(L^2(\Gamma_0),L^2(\Gamma_0))$ and 
\begin{equation*}
\|T\|_0 =\sup_{u \in L^2(\Gamma_0), u \not=0} \frac{\|Tu\|_{L^2(\Gamma_0)} }{\|u\|_{L^2(\Gamma_0)}}\leq  \|q\|_{L^{\infty}(\Gamma_0)}\leq C\|q\|_{H^s(\Gamma_0)}.
\end{equation*} 
Thus, the result follows  by interpolation (see~\cite{Bergh} or~\cite{lunardi_interp}).
\end{proof}
From Proposition~\ref{regularite_neumann}   and Lemma~\ref{mpojsdhhbg} we deduce the following result:
\begin{proposition} \label{regstat}
Let  $k \in \mathbb{N}$ and $s\in \mathbb{R}$ with $s >\frac{d-1}{2}$ and $s \geq k+\frac{1}{2}$. Assume that $\Omega$ is  of class $\mathcal{C}^{k+1,1}$. Let  $\alpha > 0$, $M >0$, $f \in H^k(\Omega)^d$, $g \in H^{k+\frac{1}{2}}(\Gamma_e)^d$ and $q \in  H^{s}(\Gamma_0)$ such that  $ q \geq \alpha$ on  $\Gamma_0$. Then the solution $(u,p)$ of system~\ref{eq3}  belongs to  $H^{k+2}(\Omega)^d \times H^{k+1}(\Omega)$. Moreover, there exists a  constant $C(\alpha,M)>0$ such that for every $q \in H^s(\Gamma_0) $ satisfying $\|q\|_{H^s(\Gamma_0) } \leq M$,
\begin{equation*}
\|u\|_{H^{k+2}(\Omega)^d} + \|p\|_{H^{k+1}(\Omega)} \leq C(\alpha,M) (\|g\|_{ H^{k+\frac{1}{2}}(\Gamma_e)^d} + \|f\|_{H^k(\Omega)^d}).
\end{equation*}
\end{proposition}
\begin{proof}[Proof of Proposition~\ref{regstat}]
Let us prove the result for $k=0$. Let $h=-qu_{|\Gamma_0}  + g$. According to Proposition~\ref{existence1}, $u$ belongs to $H^1(\Omega)^d$. We obtain from Lemma~\ref{mpojsdhhbg} for $r=1/2$ that $qu_{|\Gamma_0} \in H^{\frac{1}{2}}(\Gamma_0)^d$ , which implies,  since $g\in H^{\frac{1}{2}}(\Gamma_e)^d$  and  $\overline{\Gamma}_e \cap \overline{\Gamma}_0 = \emptyset$, that $h \in H^{\frac{1}{2}}(\partial \Omega)^d$. 
Using Proposition~\ref{regularite_neumann} with $k=0$ we obtain that $(u,p) \in H^2(\Omega)^d \times H^1(\Omega)$ and: 
\begin{equation*}
\|u\|_{{H^{2}(\Omega)}^d}+ \|p\|_{H^{1}(\Omega)}  \leq  C(\|h \|_{H^{\frac{1}{2}}(\partial \Omega)^d } + \|f\|_{L^2(\Omega)^d}).
\end{equation*}
But, since by assumption, $\|q\|_{H^{ s}(\Gamma_0) } \leq M$, we have from Lemma~\ref{mpojsdhhbg} with $r=1/2$, that:
\begin{equation*} \label{ine}
\|h\|_{H^{\frac{1}{2}}(\partial \Omega)^d  } \leq C (M)( \|u\|_{H^{\frac{1}{2}}(\partial \Omega)^d }+ \|g\|_{H^{\frac{1}{2}}(\Gamma_e)^d}).
\end{equation*}
We obtain:  
\begin{equation*}
\|u\|_{{H^{2}(\Omega)}^d}+ \|p\|_{H^{1}(\Omega)} \leq  C(M)(\|g\|_{H^{\frac{1}{2}}(\Gamma_e)^d }+ \|u \|_{H^1(\Omega)^d } + \|f\|_{L^2(\Omega)^d}).
\end{equation*}
Thus we obtain the result for $k=0$ using the inequality of Proposition~\ref{existence1}. We then proceed by induction to prove the result for any $k \in \mathbb{N}$.  
%
\end{proof}
\begin{remark}
Note that the space to which the Robin coefficient $q$ belongs is not optimal. One could surely obtain similar regularity result for a less regular Robin coefficient.  In fact, the key argument to proceed by induction in the proof of Proposition~\ref{regstat} is that  $q u_{| \Gamma_0} \in H^{k+\frac{1}{2} }(\Gamma_0)^d$, for $u \in H^{k+\frac{1}{2}}(\Gamma_0)^d$ (this property allows to apply the regularity result given by Proposition~\ref{regularite_neumann}).
\end{remark}
%
%
\subsection{Regularity of the evolution problem.}
Concerning the initial problem~\ref{eq1}, we can prove, using the Galerkin method, the following regularity results. For the sake of completeness, the proof of Theorem~\ref{reg2} is given in the appendix. 
\begin{theorem} \label{reg2}
Let $s \in \mathbb{R}$ be such that $s>\frac{d-1}{2}$, $T > 0$, $\alpha >0$ and  $u_0 \in V$. We assume that $\Omega$ is of class $\mathcal{C}^{1,1}$, $g \in H^1(0,T;H^\frac{1}{2}(\Gamma_e)^d)$ and $ q \in  H^{s}(\Gamma_0)$ is such that  $ q \geq \alpha$ on $\Gamma_0$. Then problem~\ref{eq1} admits a unique solution $(u,p) \in L^2(0,T;H^2(\Omega)^d) \cap H^1(0,T;L^2(\Omega)^d)\cap  L^{\infty}(0,T;V) \times  L^2(0,T;H^1(\Omega))$.
\end{theorem}
The following corollary will be useful when we will prove stability estimates for the evolution problem~\ref{eq1}. 
\begin{corollary} \label{regularite2}
Let  $s \in \mathbb{R}$ be such that $s>\frac{d-1}{2}$ and $s\geq \frac{3}{2}$,  $T>0$, $\alpha >0$ and  $u_0 \in H^3(\Omega)^d \cap H$. We assume that $\Omega$ is  of class $\mathcal{C}^{2,1}$, $g \in H^2(0,T;H^\frac{3}{2}(\Gamma_e)^d)$ and that $ q \in  H^s(\Gamma_0)$ is such that  $q \geq \alpha$ on $\Gamma_0$. Then, problem~\ref{eq1} admits a unique solution $(u,p) \in L^{\infty}(0,T;H^3(\Omega)^d)\cap H^1(0,T;H^2(\Omega)^d) \cap H^2(0,T;L^2(\Omega)^d) \times L^{\infty}(0,T;H^2(\Omega)) \cap H^1(0,T;H^1(\Omega))$.
\end{corollary}
\begin{proof}[Proof of Corollary~\ref{regularite2}]
Let $(u,p)$ be the solution of~\ref{eq1}. Let us consider the following system: 
\begin{equation} \label{eq31}
	\left \{   
		\begin{array}{ccll}
                \displaystyle       \partial_t v - \Delta v + \nabla \zeta		&	=	&	0,							& \textrm{ in } (0,T) \times \Omega, \\ 
                \displaystyle                          div \textrm{ } v 		&	=	&	0,							& \textrm{ in } (0,T) \times \Omega, \\  
                 \displaystyle               \frac{\partial v }{\partial n} - \zeta n		&	=	&	\partial_t g , 							& \textrm{ on } (0,T) \times \Gamma_e, \\   
                 \displaystyle               \frac{\partial v }{\partial n} - \zeta n	+ qv &	=	&	0,							& \textrm{ on } (0,T) \times \Gamma_0,  \\
                \displaystyle               v(0)					&	=	&	\Delta u_0 - \nabla p_0, 			& \textrm{ in } \Omega, 
                   \end{array}     
         \right. 
\end{equation}
where $p_0 \in H^2(\Omega)$ is defined as the solution of the following elliptic boundary problem:
\begin{equation*} 
	\left \{   
		\begin{array}{ccll}
               \displaystyle         \Delta p_0					&	=	&	\displaystyle0,														& \textrm{ in }  \Omega, \\ 
              \displaystyle                  p_0					&	=	&	\displaystyle\frac{\partial u_0 }{\partial n}  \cdot n  - g_{|t=0}\cdot n,  							& \textrm{ on } \Gamma_e, \\   
                        \displaystyle        p_0					&	=	&\displaystyle	\frac{\partial u_0 }{\partial n}  \cdot n   +	qu_0 \cdot n,							& \textrm{ on }  \Gamma_0.  \\
                   \end{array}     
         \right. 
\end{equation*}
According to Theorem~\ref{reg2},  we obtain that $(v,\zeta)$ belongs to $ L^2(0,T;H^2(\Omega)^d) \cap H^1(0,T;L^2(\Omega)^d) \cap L^{\infty}(0,T;V)\times  L^2(0,T;H^1(\Omega))$. Remark that $\displaystyle \left(\partial_t u,\partial_t p\right)$ is solution of system \ref{eq31} in the distribution sense on $(0,T)$. Thus, by uniqueness, $\displaystyle (v,\zeta)=\left(\partial_t u,\partial_t p\right)$. Then, since $q \in H^s(\Gamma_0)$ and $\displaystyle \left(\partial_t u,g\right) \in L^{\infty}(0,T;V)\times L^{\infty}(0,T;H^{\frac{3}{2}}(\Gamma_e)^d)$ we deduce from Proposition~\ref{regstat} that $(u,p) \in L^{\infty}(0,T;H^3(\Omega)^d) \times L^{\infty}(0,T;H^2(\Omega))$.
\end{proof}
%
%
%
\section{Identifiability}
\label{identifiability}
\subsection{Unique continuation}
We start by recalling a unique continuation result for the Stokes equations proved in~\cite{fabre}. 
%
%
\begin{theorem}  \label{theo:continuation}
We denote by $Q=(0,T) \times \Omega$  and  let $O$ be an open subset in $Q$. The horizontal component of $O$ is 
\begin{equation*}
C(O)=\{ (t,x) \in Q/\, \exists \, x_0 \in \Omega, (t,x_0) \in O \}.
\end{equation*}
Let $(u,p) \in L^2(0,T;H^1_{loc}(\Omega))^d \times L^2_{loc}(Q)$ be a weak solution of 
\begin{equation*}
      \left\{      
            \begin{array}{rcll}
             \displaystyle       \partial_t u - \Delta u + \nabla p   &	=	&	0,	&	\textrm{ in  } (0,T) \times \Omega, \\ 
             \displaystyle                             div \textrm{ } u &	=	&	0,	&	\textrm{ in  } (0,T) \times \Omega, \\  
              \end{array}     
         \right.
\end{equation*}
satisfying $u=0$ in  $O$ then $u=0$  and  $p$ is constant in $C(O)$.
\end{theorem}
From this theorem, we easily deduce the following result which will be useful in the next subsection.
\begin{corollary} \label{continuation unique}
Let $\delta > 0$, $x_0 \in \partial \Omega$, $t_0 \in (0,T)$ and $r >0$ be such that $\gamma = (t_0-\delta,t_0+\delta) \times( \mathcal{B}(x_0,r) \cap \partial \Omega )$ is an open set in $ (0,T) \times \partial \Omega$. Let $(u,p) \in L^2(0,T;H^2(\Omega)^d) \times L^2(0,T;H^1(\Omega))$ be solution of:
\begin{equation*}
      \left\{      
            \begin{array}{rcll}
               \displaystyle       \partial_t u  - \Delta u + \nabla p   		&	=	&	0,	&	\textrm{ in  } (0,T) \times \Omega, \\ 
               \displaystyle                		           div \textrm{ } u	&	=	&	0,	&	\textrm{ in  } (0,T) \times \Omega, \\  
              \end{array}     
         \right.
\end{equation*}
satisfying
$u=0$  and  $\displaystyle \frac{\partial u}{\partial n}-pn=0$ on $\gamma$. Then $u=0$   and  $p=0$ in $(t_0-\delta,t_0+\delta) \times \Omega$.
\end{corollary}
\begin{proof}[Proof of Corollary~\ref{continuation unique}]
We extend $u$  and  $p$ by $0$ on  $(t_0-\delta,t_0+\delta) \times (\mathcal{B}(x_0,r) \cap \Omega^c)$:
\begin{equation*} 
	 \tilde{u} \textrm{ (resp } \tilde{p} )= \left\{   
		\begin{array}{ll}
  \displaystyle                       u		\textrm{ (resp } p ),		&	 \textrm{ in  }  (t_0-\delta,t_0+\delta) \times\Omega, \\ 
  \displaystyle                       0,				&	 \textrm{ in  }  (t_0-\delta,t_0+\delta) \times (\mathcal{B}(x_0,r) \cap \Omega^c),
		\end{array}     
         \right. 
 \end{equation*}
and  we denote $\displaystyle \tilde{\Omega}=\Omega  \cup \mathcal{B}(x_0,r)$. Let us verify that $(\tilde{u},\tilde{p}) \in L^2(0,T;H^1(\Omega)^d) \times L^2(0,T;L^2(\Omega))$ is still a solution of the Stokes equations in $\tilde{\Omega}$. 
Let $v \in \mathcal{D}(\tilde{\Omega})^d$. We check by integration by parts in space that almost everywhere in $t \in  (t_0-\delta,t_0+\delta) $:	
\begin{equation*}
 \int_{\tilde{\Omega}} \partial_t \tilde{u} \cdot  v + \int_{\tilde{\Omega}} \nabla \tilde{u} : \nabla v - \int_{\tilde{\Omega}} \tilde{p}\textrm{ } div \textrm{ } v =0.
\end{equation*}
Moreover $div \textrm{ } \tilde{u}=0$ in $ (t_0-\delta,t_0+\delta) \times \tilde{\Omega}$. Therefore, we can apply Theorem~\ref{theo:continuation} to $(\tilde{u},\tilde{p})$: $(\tilde{u},\tilde{p})=(0,0)$ in $(t_0-\delta,t_0+\delta)\times \tilde{\Omega}$ which implies that $u=0$ and $p$ is constant in $(t_0-\delta,t_0+\delta)\times \Omega$. At last, the fact that $\displaystyle \frac{\partial u}{\partial n}-pn=0$ on $\gamma$ implies that $p=0$ in $(t_0-\delta,t_0+\delta)\times \Omega$.
\end{proof}
\subsection{Application}
\begin{proposition} \label{identifiabilite}
Let $s >\frac{d-1}{2}$, $T > 0$, $\alpha > 0$, $x_e \in \Gamma_e$, $r>0$, $g \in H^1(0,T;H^{\frac{1}{2}}(\Gamma_e)^d)$  be non identically zero, $u_0 \in V$  and  $q_j \in H^s(\Gamma_0)$ be such that $q_j \geq \alpha $ on $\Gamma_0$  for $j=1,2$. Let  $(u_j,p_j)$  be the weak solutions of~\ref{eq1} with $q=q_j$ for $j=1,2$.
We assume that  $u_1=u_2$ on $(0,T) \times( \mathcal{B}(x_e,r) \cap \Gamma_e)$. Then $q_1=q_2$ on $ \Gamma_0$.
\end{proposition}
\begin{proof}[Proof of Proposition~\ref{identifiabilite}]
We are going to prove Proposition~\ref{identifiabilite} by contradiction: we assume that $q_1$ is not identically equal to $q_2$ on $\Gamma_0$.
\par
Thanks to Theorem~\ref{reg2}, we have $(u_j,p_j) \in  L^2(0,T;H^2(\Omega)^d) \times  L^2(0,T;H^1(\Omega))$ for $j=1,2$. We define by $u=u_1-u_2$  and  $p=p_1-p_2$. Let us notice that $(u,p)$ is the solution of the following problem: 
\begin{equation*} 
	\left \{   
		\begin{array}{ccll}
                \displaystyle       \partial_t u - \Delta u + \nabla p							&	=	&	0,							& \textrm{ in  } (0,T) \times \Omega, \\ 
                \displaystyle                          div \textrm{ } u 						&	=	&	0,							& \textrm{ in  } (0,T) \times \Omega, \\  
                           \displaystyle    \frac{\partial u }{\partial n} - pn						&	=	&	0,							& \textrm{ on } (0,T) \times \Gamma_e, \\   
                \displaystyle               \frac{\partial u }{\partial n} - pn	+ q_1u_1-q_2 u_2		&	=	&	0,							& \textrm{ on } (0,T) \times \Gamma_0.                                                                   		\end{array}     
         \right. 
 \end{equation*}
By assumption, $u=0$ and $\displaystyle \frac{\partial u}{\partial n} -pn =0$ on $ \displaystyle (0,T) \times ( \mathcal{B}(x_e,r) \cap \Gamma_e)$. Thus, according to Corollary~\ref{continuation unique}, $u_1=u_2$  and  $p_1=p_2$ in  $(0,T) \times \Omega$. Consequently, we deduce from
%
%
\begin{equation*}
 \begin{split}
\frac{\partial u_1 }{\partial n} - p_1 n	+ q_1u_1 &	=		0,							 \text{ on } (0,T) \times \Gamma_0,       \\
\frac{\partial u_1 }{\partial n} - p_1 n	+ q_2u_1 &	=		0,							 \text{ on } (0,T) \times \Gamma_0,       
 \end{split}
\end{equation*}

that
\begin{equation} \label{eq6}
u_1(q_1 - q_2) = 0 \textrm{ on }  (0,T) \times \Gamma_0.
\end{equation}
By assumption, $q_1$  is not identically equal to $q_2$. Since $\displaystyle s>\frac{d-1}{2}$, $q_1$  and  $q_2$ are continuous on $\Gamma_0$. Thus, we can find an open set $\kappa \subset \Gamma_0$ with a positive measure such that: 
\begin{equation*}
(q_1-q_2)(x) \not= 0 \textrm{, } \forall x \in \kappa.
\end{equation*}
Equation~\ref{eq6} implies that $u_1 \equiv 0$ on $(0,T) \times \kappa$  and  then $u_1$ is the solution of 
\begin{equation*}
	\left \{   
		\begin{array}{ccll}
             \displaystyle         \partial_t u_1 - \Delta u_1 + \nabla p_1		&	=	&	0,							& \textrm{ in  } (0,T) \times \Omega, \\ 
            \displaystyle           div \textrm{ } u_1						&	=	&	0,							& \textrm{in } (0,T) \times \Omega, \\
		\displaystyle		                                          u_1 		&	=	&	0,							& \textrm{ on } (0,T) \times \kappa, \\  
              \displaystyle                  			\frac{\partial u_1 }{\partial n} - p_1n &	=	&	0,							& \textrm{ on } (0,T) \times \kappa.                                                                  		\end{array}     
         \right. 
         \qquad 
 \end{equation*}
Applying again Corollary~\ref{continuation unique}, we obtain that $u_1= 0$  and  $p_1=0$ in  $(0,T) \times \Omega$. This is in contradiction with the fact that $g$ is non identically zero. 
\end{proof}
%
%
%
%
%
\section{Stability estimates}
\label{logarithmic}
In this section, we assume that $d=2$ and that the open set $\Omega \subset \mathbb{R}^2$ is of class $\mathcal{C}^{3,1}$. 
\par
We are going to prove stability estimates for the inverse problem we are interested in by using a global Carleman inequality which is stated in Lemma~\ref{bukh}. 
\par
First, in Theorem~\ref{bebel}, we state a  stability estimate for the stationary problem. Then we deduce from this theorem two stability estimates for the evolution problem~\ref{eq1} by using an inequality coming from the analytic semigroup theory. To be more precise, we treat separately the case where  $g$ does not depend on time  (see Theorem~\ref{ok}) and  the case where $g$ depends on time (see Theorem~\ref{okok}).
\subsection{Carleman inequality}
Let us state a global Carleman inequality proved by A. L. Bukhgeim in \cite{bukhgeim}:
\begin{lemma} \label{bukh}
Let $\Psi \in \mathcal{C}^2(\overline{\Omega})$. We have:
\begin{multline} \label{eq7}
\int_{\Omega} ( \Delta \Psi |u|^2 + (\Delta  \Psi -1) |\nabla u|^2 ) e^{\Psi} \\ \leq \int_{\Omega} |\Delta u|^2 e^{\Psi} + \int_{\partial \Omega} \frac{\partial \Psi }{\partial n} \left(|u|^2+ |\nabla u|^2 +2\left |\frac{\partial |\nabla u|^2}{\partial \tau}\right| \right ) e^{\Psi}
\end{multline}
for all $u \in \mathcal{C}^2(\overline{\Omega})$.
\end{lemma}
The proof of this result, which is only valid in dimension $2$, uses computational properties of function defined on $\mathbb{C}$ (in particular, the fact that $4 \partial_{\overline{z}}\partial_{z} = \Delta$).
\begin{remark}
 The result is still true for $u \in H^3(\Omega)$. Indeed, for all $u \in H^3(\Omega)$, there exists $(u_n)_{n \in \mathbb{N}} \in \mathcal{C}^2(\overline{\Omega})^{\mathbb{N}}$ such that 
\begin{equation} \label{eq8}
u_n \to u \textrm{ in  } H^3(\Omega). 
\end{equation}
We can apply Lemma~\ref{bukh} to $u_n$, for all $n \in \mathbb{N}$. Let us prove that:
\begin{equation} \label{sol_sol}
\lim_{n \to \infty} \int_{\partial \Omega}\frac{\partial \Psi }{\partial n} \left| \frac{\partial  |\nabla u_n|^2}{\partial \tau} \right| e^{\Psi}=\int_{\partial \Omega}\frac{\partial \Psi }{\partial n}  \left| \frac{\partial  |\nabla u|^2}{\partial \tau} \right| e^{\Psi}. 
\end{equation}
Note first that $\displaystyle \int_{\partial \Omega}\frac{\partial \Psi }{\partial n}  \left| \frac{\partial  |\nabla u|^2}{\partial \tau} \right| e^{\Psi}$ has a meaning for $u \in H^3(\Omega)$: 
\begin{equation*}
\int_{\partial \Omega}\frac{\partial \Psi }{\partial n}  \left| \frac{\partial  |\nabla u|^2}{\partial \tau} \right| e^{\Psi} \leq 2\|\Psi\|_{\mathcal{C}^1(\overline{\Omega})} \|e^{\Psi}\|_{\mathcal{C}^0(\overline{\Omega})}   \left(\sum_{i=1}^2 \int_{\partial \Omega}| \nabla u| \cdot |\nabla \partial_i u| \right)< \infty.
\end{equation*}
We have: 
 \begin{multline*}
 \int_{\partial \Omega}\left |\frac{\partial \Psi }{\partial n}\right| \left| \left| \frac{\partial  |\nabla u_n|^2}{\partial \tau} \right|- \left| \frac{\partial  |\nabla u|^2}{\partial \tau} \right| \right|	e^{\Psi} \\
 \leq C\|\Psi\|_{\mathcal{C}^1(\overline{\Omega})} \|e^{\Psi}\|_{\mathcal{C}^0(\overline{\Omega})}  \sum_{i,j=1}^2 \left(\int_{\partial \Omega}| \partial_j u |^2 \right)^{\frac{1}{2}} \left(\int_{\partial \Omega}|  \partial_{ij} u_n - \partial_{ij} u|^2 \right)^{\frac{1}{2}} \\  + C\|\Psi\|_{\mathcal{C}^1(\overline{\Omega})} \|e^{\Psi}\|_{\mathcal{C}^0(\overline{\Omega})}   \sum_{i,j=1}^2\left(\int_{\partial \Omega} |\partial_{ij} u_n| ^2\right)^{\frac{1}{2}}\left(\int_{\partial \Omega} | \partial_{j} u_n-\partial_{j} u |^2\right)^{\frac{1}{2}}.
\end{multline*}
According to $\ref{eq8}$, the sequence $ \displaystyle(\partial_{ij} u_n )_{n \in \mathbb{N}}$ converges in $L^2(\partial \Omega)$ towards $\displaystyle\partial_{ij}u$ and $\|\partial_{ij}u_n\|_{L^2(\partial \Omega)}$ is bounded by a constant independent of $n$. Then, equality~\ref{sol_sol} follows from~\ref{eq8}.
\end{remark}
%
\subsection{The stationary case }
%
For the stationary problem: 
\begin{equation} \label{eq9}
	\left \{   
		\begin{array}{ccll}
             \displaystyle          - \Delta u + \nabla p			&	=	&	0,							& \textrm{ in  }  \Omega, \\ 
              \displaystyle                            div \textrm{ } u 		&	=	&	0,							& \textrm{ in  }  \Omega, \\  
              \displaystyle                 \frac{\partial u }{\partial n} - pn		&	=	&	g, 							& \textrm{ on }  \Gamma_e, \\   
             \displaystyle                  \frac{\partial u }{\partial n} - pn	+ qu &	=	&	0,							& \textrm{ on } \Gamma_0,  \\
                   \end{array}     
         \right.  
 \end{equation}
we have the following stability estimate.
%
%
\begin{theorem} \label{bebel}
Let  $\alpha >0$, $M_1> 0$, $M_2>0$, $(g,q_j) \in H^{\frac{5}{2}}(\Gamma_e)^2 \times H^{\frac{5}{2}}(\Gamma_0)$ for  $j=1,2$ be such that $g$ is not identically zero, $\|g\|_{H^{\frac{5}{2}}(\Gamma_e)} \leq M_1$, $ q_j \geq \alpha $ on $\Gamma_0$  and   $\|q_j\|_{H^{\frac{5}{2} }(\Gamma_0)}  \leq M_2$. We denote by $(u_j,p_j)$ the solution of sustem \ref{eq9} associated to $q=q_j$ for $j=1,2$. Let $K$ be a compact subset of $\{ x \in \Gamma_0 / u_1(x) \not= 0\}$ and $m >0$ be a constant such that $|u_1| \geq m$ on $K$. 
\par
Then there exist  positive constants $C(M_1,M_2,\alpha)$  and  $C_1(M_1,M_2,\alpha)$ such that
\begin{equation} \label{kjhg}
\|q_1-q_2\|_{L^2(K)} \leq 
\frac{1}{m}\frac{C(M_1,M_2,\alpha)} { \left(\ln \left( \frac{C_1(M_1,M_2,\alpha)}{\| u_1-u_2\|_{L^2(\Gamma_e)^2} + \|p_1-p_2\|_{L^2(\Gamma_e)} +  \left\| \frac{\partial p_1}{\partial n} - \frac{\partial p_2}{\partial n}\right\|_{L^2(\Gamma_e)}  }   \right)\right)^{\frac{1}{2} }}.
\end{equation}
\end{theorem}
\begin{remark}
Since $g$ is not identically zero, Corollary~\ref{continuation unique} ensures that $\{ x \in \Gamma_0 / u_1(x) \not= 0\}$ is not empty. Moreover, according to Proposition \ref{regstat}, $u_1$ is continuous on $\overline{\Omega}$, thus we obtain the existence of a compact $K$ and a constant $m$ as in Theorem~\ref{bebel}.  We notice however that the constants  involved in the estimate~\ref{kjhg} and the set $K$ depend on $u_1$. Finding a uniform lower bound for any solution $u$ of system~\ref{eq9}  remains an open question. We refer to~\cite{chaabane_jaoua},~\cite{alessandrini_rondi} and~\cite{alessandrini_sincich} for the case of the scalar Laplace equation. 
\end{remark} 
\begin{remark}
In~\cite{cheng-choulli-lin}, the same kind of inequality is proved for the Laplacian problem with Robin boundary conditions under the hypothesis that the measurements are small enough. Here, we free ourselves from this smallness assumption on the measurements.
\end{remark} 
\begin{remark}
If we compare this result with the identifiability property (Proposition~\ref{identifiabilite}), we notice that we need additional measurements on the solution. In Proposition~\ref{identifiabilite}, we only have to assume that $u_1=u_2$  and $\displaystyle \frac{\partial u_1}{\partial n} -p_1n= \frac{\partial u_2}{\partial n}-p_2n$ on $\displaystyle \Gamma \subseteq \Gamma_e$, where $\Gamma$ is a non-empty open part of the boundary, in order to get the identifiability of the Robin coefficient $\displaystyle q$ on $\displaystyle \Gamma_0$. Here, besides a measurement on $u_1-u_2$, we need measurements on $\displaystyle \frac{\partial u_1}{\partial n} - \frac{\partial u_2}{\partial n}$ (or $p_1-p_2$) and $ \displaystyle \frac{\partial p_1}{\partial n} - \frac{\partial p_2}{\partial n}$.
\end{remark} 
Let us begin by proving this intermediate result which gives us a logarithmic estimate of the traces of $u$, $\nabla u$, $p$, $\nabla p$ on $\Gamma_0$ with respect to the ones on $\Gamma_e$.
%
\begin{lemma} \label{stab}
Let $(u,p) \in H^4(\Omega)^2 \times H^3(\Omega)$ be the solution in  $\Omega$ of 
\begin{equation*}
      \left\{      
            \begin{array}{ccl}
                    - \Delta u + \nabla p   &	=	&	0,	 \\ 
                                 div \textrm{ } u &	=	&	0.	 \\  
              \end{array}     
         \right.
\end{equation*}
Then, there exist $C>0$, $C_1>0$ and $d_0>0$ such that for all $\tilde{d}>d_0$: 
\begin{multline} \label{ineg_lemma}
\| u\|_{L^2(\Gamma_{0})^2}+ \|\nabla u\|_{L^2(\Gamma_{0})^4} + \|p\|_{L^2(\Gamma_{0})} + \|\nabla p \|_{L^2(\Gamma_{0})^2}  \\
\leq \tilde{d}C \frac{\|u\|_{H^3(\Omega)^2} + \|p\|_{H^3(\Omega)}} { \left( \ln \left( C_1\tilde{d}^2 \frac{\|u\|_{H^3(\Omega)^2} + \|p\|_{H^3(\Omega)}}{\| u\|_{L^2(\Gamma_e)^2}+ \left \|\frac{\partial{u}}{\partial n}\right \|_{L^2(\Gamma_e)^2} + \|p\|_{L^2(\Gamma_e)} + \left\|\frac{\partial{p}}{\partial n}\right\|_{L^2(\Gamma_e)}  }  \right) \right)^{\frac{1}{2}} }.
\end{multline}
\end{lemma}
%
%
%
%
%
\begin{proof}[Proof of Lemma~\ref{stab}]
The proof  is based on the Carleman inequality of Lemma~\ref{bukh} for an appropriate choice of $\Psi$. Note that we will apply~\ref{eq7} twice: one time for the velocity $u$ and one time for the pressure $p$. The weight function $\Psi$ is chosen in order to estimate the traces on $\Gamma_0$ with respect to the ones on $\Gamma_e$.
\\
\\
\underline{Step 1}: choice of $\Psi$.\\
We choose $\Psi$ as in~\cite{cheng-choulli-lin}.  
 There exists $\Psi_0 \in \mathcal{C}^2(\overline{\Omega})$ non identically zero such that:
\begin{equation*}
\begin{tabular}{cccc}
$\Delta \Psi_0 = 0$ in  $\Omega$,  & $\Psi_0 = 0$ on $\Gamma_0$, & $\Psi_0 \geq 0$ on $\Gamma_e$, &$\displaystyle  \frac{\partial \Psi_0}{\partial n}< 0$ on $\Gamma_0$.
\end{tabular} 
\end{equation*}
Indeed, let $\chi \in \mathcal{C}^2({\partial \Omega})$ such that 
\begin{equation*} \label{chi_Psi_0}
\chi = 0 \textrm{ on } \Gamma_0,  \textrm{ } \chi \geq 0 \textrm{ on } \Gamma_e, 
\end{equation*}
 and  $\chi$ non identically zero on $\Gamma_e$. The boundary value problem :
\begin{equation} \label{choice_Psi_0} 
	\left \{   
		\begin{array}{rcll}
                    \displaystyle     \Delta \Psi_0				&	=	&	0,							& \textrm{ in  }  \Omega, \\ 
                    \displaystyle                      \Psi_0				&	=	&	\chi,							& \textrm{ on } \partial \Omega, 
		\end{array}     
         \right. 
 \end{equation}
has a unique solution $\Psi_0 \in \mathcal{C}^2(\overline{\Omega} )$. Note that $\Psi_0$ is not constant because $\chi$ is non identically zero. So, from the strong maximum principle, $\Psi_0 > 0$ in  $\Omega$. According to Hopf Lemma, we have $ \displaystyle \frac{\partial \Psi_0}{\partial n}<0 $ on $\Gamma_0$. 
\par
Let $\lambda>0$. We denote by $\Psi_1 \in \mathcal{C}^2(\overline{\Omega})$ the unique solution of the boundary value problem: 
\begin{equation*} 
	\left \{   
		\begin{array}{rcll}
              \displaystyle           \Delta \Psi_1				&	=	&	\lambda,						& \textrm{ in  }  \Omega, \\ 
             \displaystyle                             \Psi_1				&	=	&	0,							& \textrm{ on } \partial \Omega. 
		\end{array}     
         \right. 
 \end{equation*}
From the comparison principle  and  the strong maximum principle, we have $\Psi_1 < 0$ in  $\Omega$. Moreover, according to the Hopf Lemma, we have $ \displaystyle \frac{\partial \Psi_1}{\partial n}>0$ on $\partial \Omega$. \\
\\
Let us consider $\Psi= \Psi_1 + s \Psi_0$, for  $s>0$. To summarize, the function $\Psi$ has the following properties: 
\begin{equation*}
\begin{array}{cccc}
 \Delta \Psi =\lambda \textrm{ in } \Omega, & \Psi =0 \textrm{ on } \Gamma_0, & \Psi  \geq 0 \textrm{ on } \Gamma_e, \textrm{ and } & \displaystyle s  \frac{\partial \Psi_0}{\partial n} \leq \frac{\partial \Psi}{\partial n} \leq  \frac{\partial \Psi_1}{\partial n} \textrm{ on } \Gamma_0.
\end{array}
\end{equation*}
\\
\underline{Step 2}: 
We first apply Lemma~\ref{bukh} to $u$. Using the fact that $\displaystyle \Delta u = \nabla p$, we have: 
\begin{equation} \label{eq11}
\begin{split}
\int_{\Omega} ( \Delta \Psi |u|^2 + &(\Delta  \Psi -1) |\nabla u|^2 ) e^{\Psi} \\ & \leq        \int_{\Omega} |\nabla p|^2 e^{\Psi} 
 + \int_{\partial \Omega} \frac{\partial \Psi }{\partial n} \left(|u|^2+ |\nabla u|^2 +2 \left| \frac{\partial  |\nabla u|^2}{\partial \tau} \right|\right ) e^{\Psi}.
\end{split}
\end{equation}
Then, we  apply once again Lemma~\ref{bukh}   to $p$:
\begin{equation}  \label{eq12}
\begin{split}
\int_{\Omega} ( \Delta \Psi |p|^2 + &(\Delta  \Psi -1) |\nabla p|^2 ) e^{\Psi}\\ &\leq \int_{\Omega} |\Delta p|^2 e^{\Psi} 
+ \int_{\partial \Omega} \frac{\partial \Psi }{\partial n}\left(|p|^2+ |\nabla p|^2 +2  \left|\frac{\partial |\nabla p|^2}{\partial \tau}\right| \right) e^{\Psi}.
\end{split}
\end{equation}
We have $\Delta p =$ div$(\Delta u)=0$ hence  $\displaystyle \int_{\Omega} |\Delta p|^2 e^{\Psi} =0$. We now choose $\lambda \geq 2$.  By summing up inequalities~\ref{eq11}  and~\ref{eq12}  and  by eliminating the integrals on $\Omega$ in the left hand side which are positive terms, we obtain:

\begin{multline*}
\int_{\partial \Omega} \frac{\partial \Psi }{\partial n}\left(|u|^2+ |\nabla u|^2 +2 \left| \frac{\partial  |\nabla u|^2}{\partial \tau} \right| \right) e^{\Psi} \\
+  \int_{\partial \Omega} \frac{\partial \Psi }{\partial n} \left(|p|^2+ |\nabla p|^2 +2  \left|\frac{\partial |\nabla p|^2}{\partial \tau}\right| \right) e^{\Psi} \geq 0.
\end{multline*}
We now specify the dependence with respect to $s$. We denote by $\displaystyle \theta=\min_{\Gamma_0}\left| \frac{\partial \Psi_0}{\partial n}\right|$. We note that on $\Gamma_0$, $e^{\Psi}=1$. Consequently: 
\begin{equation} \label{ben}
\begin{split}
- & s \theta \int_{\Gamma_0}  (|u|^2+ |\nabla u|^2 +|p|^2+ |\nabla p|^2 ) 
+  \int_{\Gamma_0}\frac{\partial \Psi_1}{\partial n} \left(|u|^2+ |\nabla u|^2 +|p|^2+ |\nabla p|^2 \right)  \\
+  &2\int_{\Gamma_0}\frac{\partial \Psi }{\partial n}\left(  \left|\frac{\partial |\nabla p|^2}{\partial \tau}\right|  +  \left| \frac{\partial  |\nabla u|^2}{\partial \tau} \right|  \right)
+  2\int_{\Gamma_e} \frac{\partial \Psi }{\partial n}\left( \left|\frac{\partial |\nabla p|^2}{\partial \tau}\right|  +  \left| \frac{\partial  |\nabla u|^2}{\partial \tau} \right| \right)e^{\Psi} \\
+ &  \int_{\Gamma_e} \frac{\partial \Psi }{\partial n} (|u|^2+ |\nabla u|^2 +|p|^2+ |\nabla p|^2 ) e^{\Psi} \geq 0 .
\end{split}
\end{equation}
Let us study each of the terms. We have:
\begin{equation*}
\int_{\Gamma_0}\frac{\partial \Psi_1}{\partial n}  (|u|^2+ |\nabla u|^2 +|p|^2+ |\nabla p|^2 )  \leq  C (\|u\|^2_{H^3(\Omega)^2} +\|p\|^2_{H^3(\Omega)}).
\end{equation*}
Moreover, since  $\displaystyle  \frac{\partial \Psi }{\partial n} \leq \frac{\partial \Psi_1}{\partial n} $ on $\Gamma_0$, we obtain:
\begin{equation*}
2 \int_{\Gamma_0} \frac{\partial \Psi }{\partial n} \left( \left|\frac{\partial |\nabla p|^2}{\partial \tau}\right|  +  \left| \frac{\partial  |\nabla u|^2}{\partial \tau} \right| \right)\leq C (\|u\|^2_{H^3(\Omega)^2} +\|p\|^2_{H^3(\Omega)}).
\end{equation*}
Since, on $\Gamma_e$, $\displaystyle \left|\frac{\partial \Psi }{\partial n}\right| \leq s C$ for $s \geq 1$,  we have:
\begin{equation*}
\int_{\Gamma_e} \frac{\partial \Psi }{\partial n} (|u|^2+ |\nabla u|^2 +|p|^2+ |\nabla p|^2 ) e^{\Psi}  \leq C s \int_{\Gamma_e}  (|u|^2+ |\nabla u|^2 +|p|^2+ |\nabla p|^2 ) e^{\Psi}.
\end{equation*}
Using Cauchy-Schwarz inequality, we obtain:
\begin{multline*}
2 \int_{\Gamma_e}   \frac{\partial \Psi }{\partial n}\left(  \left|\frac{\partial |\nabla p|^2}{\partial \tau}\right|  +  \left| \frac{\partial  |\nabla u|^2}{\partial \tau} \right|\right)e^{\Psi} \\ \leq s C(\|u\|_{H^3(\Omega)^2}+\|p\|_{H^3(\Omega)})\left( \int_{\Gamma_e} (|\nabla p|^2+|\nabla u|^2)e^{2 \Psi}\right)^{\frac{1}{2}}.
\end{multline*}
Note that $e^{\Psi}$ depends on $s$ on $\Gamma_e$. Hence, reassembling these inequalities, inequality~\ref{ben} becomes: 
\begin{equation} \label{plgrc}
\theta \int_{\Gamma_0}  (|u|^2+ |\nabla u|^2 +|p|^2+ |\nabla p|^2 )
\leq  C \left(K_s  + \frac{1}{s}(\|u\|^2_{H^3(\Omega)^2} +\|p\|^2_{H^3(\Omega)})\right),
\end{equation}
where 
\begin{multline} \label{def-Ks}
 K_s=(\|u\|_{H^3(\Omega)^2}+\|p\|_{H^3(\Omega)}) \left( \int_{\Gamma_e} (|\nabla p|^2+|\nabla u|^2)e^{2 \Psi}\right)^{\frac{1}{2}} \\ +\int_{\Gamma_e}  (|u|^2+ |\nabla u|^2 +|p|^2+ |\nabla p|^2 ) e^{\Psi} . 
\end{multline}
In order to study the dependence with respect to $s$ of $K_s$, we define $B$ by:
\begin{equation} \label{mesure}
B = \|u\|_{L^2(\Gamma_e)^2}+\|p\|_{L^2(\Gamma_e)}+\left \|\frac{\partial u}{\partial n} \right\|_{L^2(\Gamma_e)^2}+ \left \|\frac{\partial p}{\partial n} \right\|_{L^2(\Gamma_e)}.
\end{equation}
Let us estimate the first term in the expression of $K_s$. 
Remark that, thanks to classical interpolation inequalities (see~\cite{adams_fournier}),  there exists $C>0$ such that for all $f \in H^2(\Gamma_e)$:
\begin{equation*} 
\left \|\frac{\partial f}{\partial \tau} \right\|_{L^2(\Gamma_e)} \leq \|f  \|_{H^1(\Gamma_e)} \leq C \|f  \|^{\frac{1}{2}}_{L^2(\Gamma_e)} \|f  \|^{\frac{1}{2}}_{H^2(\Gamma_e)} .
\end{equation*}
Applying the previous inequality, there exists $C>0$ such that:
\begin{equation} \label{titi}
\displaystyle \int_{\Gamma_e }\left |\frac{\partial u}{\partial \tau} \right|^2  \leq C \|u\|_{H^3(\Omega)^2} \|u\|_{L^2(\Gamma_e)^2}  \textrm{ and }   \displaystyle \int_{\Gamma_e}\left |\frac{\partial p}{\partial \tau} \right|^2 \leq C\|p\|_{H^3(\Omega)} \|p\|_{L^2(\Gamma_e)} .
\end{equation}
We obtain, using the fact that $\displaystyle \nabla v=\frac{\partial v }{\partial n}n +\frac{\partial v }{\partial \tau} \tau $ on $\partial \Omega$ for all $v \in H^2(\Omega)$, and thanks to inequality~\ref{titi}, 
 that there exists $C>0$ such that:

\begin{equation*}
\begin{split}
 \int_{\Gamma_e} (|\nabla p|^2+|\nabla u|^2) e^{2 \Psi}   \leq &e^{2 ks}\left(\left\|\frac{\partial p}{\partial n}\right\|_{L^2(\Gamma_e)}^2+\left\|\frac{\partial u}{\partial n}\right\|_{L^2(\Gamma_e)^2}^2Ê\right) \\ 
& +e^{2 ks}\left(C\|u\|_{H^3(\Omega)^2} \|u\|_{L^2(\Gamma_e)^2}+C\|p\|_{H^3(\Omega)} \|p\|_{L^2(\Gamma_e)}\right   )
\\
  \leq  &Ce^{2ks} \|u\|_{H^3(\Omega)^2}  \left(\left\|\frac{\partial u}{\partial n}\right\|_{L^2(\Gamma_e)^2}+ \|u\|_{L^2(\Gamma_e)^2}\right   )
\\
& +   Ce^{2ks}\|p\|_{H^3(\Omega)}  \left(\left\|\frac{\partial p}{\partial n}\right\|_{L^2(\Gamma_e)}\|p\|_{L^2(\Gamma_e)}\right   )\\ 
 \leq & Ce^{ks}\left(\|u\|_{H^3(\Omega)^2}+\|p\|_{H^3(\Omega)}\right) B,
\end{split}
\end{equation*}
where $\displaystyle k=\max_{\Gamma} \Psi_0$. 
Similarly, for the second term in the expression of $K_s$ we prove that
\begin{equation*} 
\begin{split}
 \int_{\Gamma_e}  (|u|^2+ |\nabla u|^2 +|p|^2+ |\nabla p|^2 ) e^{\Psi} 
\leq & Ce^{ks}\left(\|u\|_{H^3(\Omega)^2}+\|p\|_{H^3(\Omega)}\right) B \\\leq & Ce^{ks}\left(\|u\|_{H^3(\Omega)^2}+\|p\|_{H^3(\Omega)}\right)^{\frac{3}{2}} B^{\frac{1}{2}}.
\end{split}
\end{equation*}
Thus, using the two previous inequalities, according to the definition \ref{def-Ks} of $K_s$, we obtain, 
\begin{equation*} 
K_s \leq C e^{ks}   \left(\|u\|_{H^3(\Omega)^2}+\|p\|_{H^3(\Omega)}\right)^{\frac{3}{2}} B^{\frac{1}{2}}. 
\end{equation*}
%
%
%
Let us denote by
 \begin{equation*}
A=\|u\|_{H^3(\Omega)^2}+\|p\|_{H^3(\Omega)}.
\end{equation*}
 Hence we get from~\ref{plgrc}:
\begin{equation*}
\int_{\Gamma_{0}}  (|u|^2+ |\nabla u|^2 +|p|^2+ |\nabla p|^2 )
\leq  C \left(A^{\frac{3}{2}} e^{ks}B^{\frac{1}{2}}  + \frac{A^2}{s} \right),
\end{equation*}
for all $ s \geq 1$. Remark that this inequality is trivially verified for $0< s \leq 1$ by continuity of the trace mapping. Let $\tilde{d}\geq1$. To summarize, we have proved that: 
\begin{equation*}
\int_{\Gamma_{0}}  (|u|^2+ |\nabla u|^2 +|p|^2+ |\nabla p|^2 )
\leq  C A^{\frac{3}{2}}  \left(e^{ks}B^{\frac{1}{2}}  + \frac{ \tilde{d}A^{\frac{1}{2}}}{s}    \right) \textrm{, } \forall s >0 .
\end{equation*}
We now optimize the upper bound with respect to $s$. We denote by 
\begin{equation*}
f(s) =e^{ks}B^{\frac{1}{2}}  +\frac{\tilde{d} A^{\frac{1}{2}}}{s}.
\end{equation*} 
Let us study the function $f$ in $\mathbb{R^*_+}$. We have:
\begin{equation*}
\left \{
	\begin{tabular}{c}
		$\lim_{s \to 0} f(s)= +  \infty $,   \\
		$\lim_{s \to \infty} f(s)= +  \infty $.
	\end{tabular}
\right. 
\end{equation*}
So since $f$ is continuous on $\mathbb{R}^+_*$, f reaches its minimum at a point $s_0>0$. At this point, 
\begin{equation*}
 f'(s_0)=0 \Leftrightarrow B^{\frac{1}{2}} = \frac{e^{-ks_0}\tilde{d}A^{\frac{1}{2}} }{k{s_0}^2}, \textrm{ thus } f(s_0)=\frac{\tilde{d}A^{\frac{1}{2}}}{k s_0^2}+ \frac{\tilde{d}A^{\frac{1}{2}}}{s_0}.
\end{equation*}
Hence:
\begin{equation*}
	\int_{\Gamma_{0}}  (|u|^2+ |\nabla u|^2 +|p|^2+ |\nabla p|^2 ) 	 \leq  \frac{C\tilde{d} A^2}{s_0^\beta} \left(\frac{1}{k}  + 1\right),
\end{equation*}
where $\beta =1$ if $s_0 \geq 1$ and $\beta=2$ otherwise. But, we notice that
\begin{equation*}
\frac{1}{B^{\frac{1}{2}}} =\frac{ k {s_0}^2 e^{k s_0}}{\tilde{d}A^{\frac{1}{2}}} \leq\frac{ k e^{ (k+2) s_0}}{\tilde{d}A^{\frac{1}{2}}},
\end{equation*}
that is to say: 
\begin{equation*}
\frac{1}{s_0} \leq \frac{k+2}{\ln\left(\frac{\tilde{d} A^{\frac{1}{2}}}{k B^{\frac{1}{2}}}\right)},
\end{equation*}
if $\tilde{d}$ is larger than a constant which only depends on $k$ and on the continuity constants of the trace mapping.
In the same way, when $s_0 < 1$ , we obtain: 
\begin{equation*}
\frac{1}{s_0^2} \leq \frac{1}{\ln\left(\frac{\tilde{d} A^{\frac{1}{2}}}{k e^k B^{\frac{1}{2}}}\right)},
\end{equation*}
if $\tilde{d}$ is larger than a constant which only depends on $k$ and on the continuity constants of the trace mapping.
Using the fact that $\ln\left(x^{\frac{1}{2}}\right)= \frac{1}{2} \ln(x)$ for all $x>0$ and according to the definition~\ref{mesure} of $B$, the desired result follows.
\end{proof}
%
%
%
%
%
%
%
\begin{remark}

Let $\Gamma \subseteq \Gamma_e$ be a non-empty open part of $\Gamma_e$. Inequality \ref{ineg_lemma} of  Lemma~\ref{stab} still holds if we replace $\Gamma_e$ by $\Gamma$ in the right-hand side. To prove this, it is sufficient in the definition \ref{choice_Psi_0}  of $\Psi_0$  to define 
$\chi \in \mathcal{C}^2({\partial \Omega})$ such that $\chi = 0 \textrm{ on } \Gamma_0\cup (\Gamma_e \setminus\overline{\Gamma})$ and  $\chi \geq 0 \textrm{ on } \Gamma.$
\end{remark}
%
%
%
%
Let us now prove Theorem~\ref{bebel}.
\begin{proof}[Proof of Theorem~\ref{bebel}]
Since $g\in H^{\frac{5}{2}}(\Gamma_e)^2$ and $q_j \in H^{\frac{5}{2}}(\Gamma_0)$ for $j=1,2$, thanks to Proposition~\ref{regstat} applied for $k=2$, there exists $C(\alpha,M_1,M_2)>0$ such that:
\begin{equation}\label{maj2222}
\|u_j\|_{H^{4}(\Omega)^2} + \|p_j\|_{H^{3}(\Omega)} \leq C(\alpha,M_1,M_2), \text{ for } j=1,\,2.
\end{equation}
In the following, we denote by $u=u_1-u_2$ and $p=p_1-p_2$. We have:
\begin{equation} \label{jqlsk}
(q_2-q_1)u_1 = q_2 u+ \frac{\partial u}{\partial n} -pn \textrm{, on } \Gamma_{0}. 
\end{equation}
Consequently, since $|u_1|  \geq m>0$ on $K$:
\begin{equation} \label{what_what_what}
\|q_1-q_2\|_{L^2(K)} \leq \frac{1}{m} C(M_2) \left( \|u\|_{L^2(\Gamma_0)^2} + \left\|\frac{\partial u}{\partial n}\right\|_{L^2(\Gamma_0)^2} + \| p\|_{L^2(\Gamma_0)} \right).
\end{equation}
Let us denote by 
\begin{equation*}
A=\|u\|_{H^3(\Omega)^2}+\|p\|_{H^3(\Omega)}
\end{equation*}
and 
\begin{equation*}
B=\| u\|_{L^2(\Gamma_e)^2}+ \|p\|_{L^2(\Gamma_e)} + \left\|\frac{\partial{u}}{\partial n}\right\|_{L^2(\Gamma_e)^2}  + \left\|\frac{\partial{p}}{\partial n}\right\|_{L^2(\Gamma_e)}.
\end{equation*} 
By applying Lemma~\ref{stab}, we obtain that there exists there exists $C(M_2)>0$, $C_1>0$  and   $d_0>0$ such that for all $\tilde{d}>d_0$: 
\begin{equation} \label{ineq-int2}
\|q_1-q_2\|_{L^2(K)} \leq \frac{\tilde{d}C(M_2)}{m}   \frac{A} { \left( \ln \left( C_1\tilde{d}^2 \frac{A}{B  }  \right) \right)^{\frac{1}{2}} }.
\end{equation}

We are going to concude the proof by studying the variation of the function defined by  $\displaystyle f_y(x)=\frac{x}{\left(\ln\left(\frac{x}{y}\right)\right)^{\frac{1}{2}}} \textrm{ on } (y,+ \infty)$, for $\displaystyle y=\frac {B}{C_1\tilde{d}^2}$.
We have 
\begin{equation*}
f'_y(x) = \frac{\ln\left(\frac{x}{y}\right)-\frac{1}{2}}{\ln\left(\frac{x}{y}\right)^{\frac{3}{2}}}.
\end{equation*}
Let us denote by $\displaystyle x_0=y e^{\frac{1}{2}}$. The function $f_y$ is decreasing on $(y,x_0]$ and is increasing on $[x_0,+\infty)$. For $\tilde{d}$ large enough, we have $\displaystyle A \geq x_0$ by continuity of the trace mapping. Using~\ref{maj2222} and since $f_y$ is increasing on $[x_0,+\infty)$, we directly deduce that $ \displaystyle f_{y}(A) \leq f_{y}(C(\alpha,M_1,M_2))$.

Using this result in inequality \ref{ineq-int2}, we get that  there exist constants $C(\alpha, M_1,M_2) >0$ and $C_1(\alpha, M_1,M_2)>0$ such that:   
\begin{equation*}  
\|q_1-q_2\|_{L^2(K)}  \leq   \frac{1}{m} \frac{ C(\alpha,M_1,M_2) }{ \left( ln \left(  
\frac{C_1(\alpha, M_1,M_2) } {\|u\|_{L^2(\Gamma_e)^2} + \|p\|_{L^2(\Gamma_e)} + \left\|\frac{\partial u }{\partial n}\right\|_{L^2(\Gamma_e)^2} +\left\|\frac{\partial p }{\partial n}\right\|_{L^2(\Gamma_e)} }  \right) \right)^{\frac{1}{2}} }.
\end{equation*} 
Since $\displaystyle \frac{\partial u}{\partial n}=pn $ on $\Gamma_e$, we obtain the desired inequality. 
\end{proof}
\begin{remark} \label{particular_case}
Note that the assumption that $|u_1|\geq m>0$ on $K$ is essential to pass from~\ref{jqlsk} to~\ref{what_what_what}. Outside the set $K$, an estimate of $q_1-q_2$ may be undetermined or highly unstable. In particular, an estimate of the Robin coefficients on the whole set $\Gamma_0$ might be worst than of logarithmic type (see~\cite{jbalia}). 
\par
Note however that for a simplified problem, it is in fact possible to obtain a logarithmic stability estimate on the whole set $\Gamma_0$ which does not depend on a given reference solution. 
 Assume that $g=g_e n$ and $q \in \mathbb{R}$ are such that
\begin{enumerate}
\item[(A)]  $\displaystyle g_e \in \mathbb{R}$ satisfies $\beta \leq g_e \leq M_1$,
\item[(B)]  $ \displaystyle \alpha \leq q \leq M_2$,
\end{enumerate}
for some $\alpha >0$, $\beta>0$, $M_1>0$ and $M_2>0$. \\
We denote by $(u_{g_e,q},p_{g_e,q})$ the solution of system~\ref{eq9} associated to $q$ and $g=g_en$. Thanks to the weak formulation of the problem, $\int_{\Gamma_e} u_{g,q} \cdot n >0$. Moreover, one can prove by contradiction and thanks to the continuity of the solution with respect to the data that there exists $m_1>0$ which depends on $M_1$, $M_2$, $\alpha$ and $\beta$ such that for all $(g_e,q) \in \mathbb{R}^2$ which satisfies $(A)$ and $(B)$, 
 $$\int_{\Gamma_e} u_{g_e,q} \cdot n \geq m_1.$$
\par
For $i=1,2$, let $q_i \in \mathbb{R}$ satisfy the assumption $(B)$ above. We define by $(u_i,p_i)=(u_{g_e,q_i},p_{g_e,q_i})$  the solution of system~\ref{eq9} associated with $g=g_en$ and $q=q_i$ for i=1,2. %
If we multiply ~\ref{jqlsk} by the unit normal $n$ and we  integrate on $\Gamma_0$, we obtain: 
\begin{equation*}
(q_2-q_1) \int_{\Gamma_0 }u_1 \cdot n  = q_2 \int_{\Gamma_0 }( u_1-u_2 )\cdot n+\displaystyle \int_{\Gamma_0 }\left( \frac{\partial u_1}{\partial n} -\frac{\partial u_2}{\partial n} \right) \cdot n-\int_{\Gamma_0}(p_1-p_2).
\end{equation*}
Since $u_1$ is divergence free, $\displaystyle \left |\int_{\Gamma_0 }u_1 \cdot n \right |=\left |\int_{\Gamma_e }u_1 \cdot n \right | \geq m_1$. Thus, we get
\begin{equation*}
\begin{split}
|q_1&- q_2| \\ \leq &C(M_1,M_2,\alpha, \beta) \left( \left \|u_1-u_2 \right\|_{L^2(\Gamma_0)^2} +\left \| \frac{\partial u_1}{\partial n} -\frac{\partial u_2}{\partial n}\right \|_{L^2(\Gamma_0)^{2}} + \| p_1-p_2 \|_{L^2(\Gamma_0)} \right).
\end{split}
\end{equation*}
We conclude as in the proof of Theorem~\ref{bebel} and obtain that  positive constants $C(M_1,M_2,\alpha)$  and  $C_1(M_1,M_2,\alpha,\beta)$ such that
\begin{equation*}
|q_1-q_2| \leq 
\frac{C(M_1,M_2,\alpha, \beta)} { \left(\ln \left( \frac{C_1(M_1,M_2,\alpha)}{\| u_1-u_2\|_{L^2(\Gamma_e)^2}+ \|p_1-p_2\|_{L^2(\Gamma_e)} +  \left\| \frac{\partial p_1}{\partial n} - \frac{\partial p_2}{\partial n}\right\|_{L^2(\Gamma_e)}  }   \right)\right)^{\frac{1}{2} }}.
\end{equation*}
\end{remark}
%
%
%
\subsection{Evolution problem}
In order to use semigroup properties, we begin by introducing the Stokes operator associated with the Robin boundary conditions on $\Gamma_0$. 
\subsubsection{Properties of the Stokes operator}
We recall that the bilinear form $a_q$ is defined by~\ref{eq5}. 
\begin{definition} \label{operator_A_q}
We define the set $\mathcal{D}(A_q)$ as follows:
\begin{equation*}
\mathcal{D}(A_q)= \{ u \in V / \exists C>0, \forall v \in V, |a_q(u,v)| \leq C \|v\|_{ L^2(\Omega)^2} \},
\end{equation*}
and the operator $A_q: \mathcal{D}(A_q) \subset H \to H$ by: 
\begin{equation*}
\forall u \in \mathcal{D}(A_q) \textrm{, } a_q(u,v)=(A_q u,v)_{ L^2(\Omega)^2}, \forall v \in V.
\end{equation*}
\end{definition}
\begin{proposition} \label{valeur_propre}
Let $\alpha >0$ and $q \in  L^{\infty}(\Gamma_0)$ such that $q \geq \alpha$ almost everywhere on $\Gamma_0$. The operator $A_q$ has the following properties: 
\begin{enumerate}
\item $A_q \in \mathcal{L}(\mathcal{D}(A_q),H)$ is invertible and its inverse is compact on $H$.
\item $A_q$ is selfadjoint.
\end{enumerate}
As a consequence, $A_q$ admits a family of eigenvalues $\phi_q^l$
\begin{equation*}
A_q \phi_q^l = \lambda_q^l \phi_q^l \textrm{ with } 0 <    \lambda_q^1 \leq \lambda_q^2 \leq... \leq \lambda_q^j \textrm{ and } \lim_{j \to \infty} \lambda_q^j= + \infty, 
\end{equation*}
which is complete and orthogonal both in $H$ and $V$.
\end{proposition}
\begin{proof}[Proof of Proposition~\ref{valeur_propre}]
It relies on classical arguments for which we refer to~\cite{brezis} or~\cite{raviart_thomas}.
\end{proof}
\begin{remark} \label{eq18b}
Let $\alpha>0$. There exists a constants $\mu>0$ such that for all $q \in L^{\infty}(\Gamma_{0})$ such that $q \geq \alpha$, for $l \in \mathbb{N}$:
\begin{equation} \label{eq18}
\lambda_q^l \geq \mu.
\end{equation}
Indeed, $\lambda_q^l \geq \lambda_q^1 = (A_q \phi_q^1,\phi_q^1)_{L^2(\Omega)^2}=a_{q}(\phi_q^1,\phi_q^1) \geq a_{\alpha}(\phi_q^1,\phi_q^1) \geq \mu \|\phi_q^1\|^2_{L^2(\Omega)^2}=\mu$, where $\mu$ is the coercivity constant associated with the bilinear form $a_{\alpha}$. 
\end{remark}
\begin{proposition} \label{isometry}
The operator $A_q^{\frac{1}{2}}: (V,a_q(.,.)^{\frac{1}{2}}) \to (H,\|$ $\|_{L^2(\Omega)^2})$ is an isometry.
\end{proposition}
\begin{proposition} \label{semigroupe}
Let $\alpha > 0$ and $q \in  L^{\infty}(\Gamma_0)$ be such that $q \geq \alpha$ almost everywhere on $\Gamma_0$. The operator $-A_q$ generates an analytic semigroup on $H$. This analytic semigroup is explicitly given by: 
\begin{equation} \label{eq18bis}
e^{-tA_q} f = \sum_{ l \geq 1} e^{-t \lambda_q^l} (\phi_q^l,f)_{L^2(\Omega)^2}\phi^l_q,
\end{equation}
for all $f \in H$.
\end{proposition}
\begin{proof}[Proof of Proposition~\ref{semigroupe}]
It follows from the construction of the operator $A_q$. We refer to ~\cite{pazy} and~\cite{dautray_lions} for details.
\end{proof}
\begin{proposition} \label{regularite}
Let $\alpha>0$, $M>0$, $k \in \mathbb{N}$ and  $s \in \mathbb{R}$ be such that  $s>\frac{d-1}{2}$ and $s \geq \frac{1}{2}+k$. We assume that $\Omega$ is of class $\mathcal{C}^{k+1,1}$ and that  $q \in  H^s(\Gamma_0)$  is such that $q \geq \alpha$ on $\Gamma_0$. 
\par
Then for each $f \in H \cap H^k(\Omega)^2 $, there exists $u \in H^{k+2}(\Omega)^2$ solution of $A_qu=f$ if and only if there exists $p \in H^{k+1}(\Omega)$ such that $(u,p)$ is solution of the following problem:  
\begin{equation} \label{mem56}
      \left\{
      \begin{array}{ccll}
            - \Delta u  + \nabla p				&	=	&	f, 	& \textrm{ in } \Omega, \\
            div \textrm{ } u 						&	=	&	0,	& \textrm{ in } \Omega, \\
     	   \frac{\partial u }{\partial n}-pn				&	=	&	0,	& \textrm{ on } \Gamma_e, \\
           \frac{\partial u }{\partial n} -p n +qu				&	=	&	0,	& \textrm{ on } \Gamma_0.
     \end{array}
     \right.
\end{equation}
Moreover, there exists a constant $C(\alpha,M)>0$ such that for every $q \in H^s(\Gamma_0)$ satisfying  $\|q\|_{H^s(\Gamma_0) } \leq M$:
\begin{equation*}
\|u\|_{H^{2+k}(\Omega)^2} \leq  C(\alpha,M) \|f\|_{H^k(\Omega)^2}.
\end{equation*}
\end{proposition}
\begin{proof}[Proof of Proposition~\ref{regularite}]
This result follows from the construction of the operator $A_q$ and from Proposition~\ref{regstat}.
\end{proof}
\begin{corollary} \label{magnifique}
Let $\alpha>0$, $k \in \mathbb{N}^*$ and  $s \in \mathbb{R}$ be such that  $s>\frac{d-1}{2}$ and $s \geq \frac{1}{2}+2(k-1)$. We assume that $\Omega$ is of class $\mathcal{C}^{2k-1,1}$ and that   $q \in  H^s(\Gamma_0)$ is such that $q \geq \alpha$ on $\Gamma_0$. 
\par
Then $\mathcal{D}(A_q^k) \hookrightarrow  H^{2k}(\Omega)^2 \cap H$.
\end{corollary}
\begin{proof}[Proof of Corollary~\ref{magnifique}]
For $k=1$, it is clear. Take now $k = 2$. Let $u \in \mathcal{D}(A_q^2)$. We have 
\begin{equation*}
 A^2_qu = f \Leftrightarrow
 	\left    \{
		\begin{array}{c}
                   A_q u =v  \\
                   A_q v=f         
                   \end{array}     
         \right. 
 \end{equation*}
But $v \in \mathcal{D}(A_q) \subset H^2(\Omega)^2 \cap H$ by assumption, so $u \in H^4(\Omega)^2 \cap H$ thanks to the regularity properties of the solution of the Stokes problem summarize in Proposition~\ref{regstat}. We conclude by induction on $k$.
\end{proof}
\begin{remark}
Let us remark that, due to the prescribed boundary conditions, $\mathcal{D}(A_q)$ is not equal to $H^2(\Omega)^2 \cap H$. 
\end{remark}
\subsubsection{The flux $g$ does not depend on $t$}
In this paragraph, we consider the evolution problem~\ref{eq1} given in the introduction. We assume in this part that $ g $ does not depend on time. Let $\alpha>0$, $M_1>0$ and $M_2>0$. In the following, we assume that 
\begin{equation} \label{eq13}
g \in H^{\frac{5}{2}}(\Gamma_e)^2 \textrm{ is non identically zero and } \|g\|_{H^{ \frac{5}{2}}(\Gamma_e)^2 } \leq M_1,
\end{equation}
\begin{equation} \label{eq14}
q \in  H^{\frac{5}{2}}(\Gamma_0) \textrm{ is such that }  \|q\|_{H^{\frac{5}{2}}(\Gamma_0) } \leq M_2 \textrm{ and } q \geq \alpha \text{ on } \Gamma_0.
\end{equation}
Let us prove the following theorem: 
\begin{theorem} \label{ok}
Let $\alpha >0$, $M_1> 0$, $M_2>0$ and $u_0 \in H \cap H^3(\Omega)^2$. We assume that $g$ satisfies~\ref{eq13} and that $q_j$ satisfies~\ref{eq14} for $j=1,2$. We denote  by $(u_j,p_j)$ the solution of system \ref{eq1} associated to $q=q_j$, for $j=1,2$. Let $K$ be a compact subset of $\{ x \in \Gamma_0 / v_1(x) \not= 0\}$, where  $(v_1,\zeta_1)$ is the solution of system \ref{eq9} with $q=q_1$ and let $m >0$ be a constant such that $|v_1| \geq m$ on $K$. Then, there exist $C(\alpha,M_1,M_2)>0$ and $C_1(M_1,M_2,\alpha)>0$ such that 
\begin{equation*}
\begin{split}
&\|q_1-q_2\|_{L^2(K)}  \leq 
\\ &\frac{1}{m}\frac{C(\alpha,M_1,M_2)}  {\left( \ln \left ( \frac{ C_1(M_1,M_2,\alpha)}{  \| u_1-u_2\|_{L^{\infty}(0,+\infty;L^2(\Gamma_e)^2)}+ \|p_1-p_2\|_{L^{\infty}(0,+\infty;L^2(\Gamma_e) )} +\left\| \frac{\partial p_1}{\partial n} -\frac{\partial p_2}{\partial n} \right \|_{L^{\infty}(0,+\infty;L^2(\Gamma_e))}  }  \right )\right) ^\frac{1}{2}}.
\end{split}
\end{equation*}
\end{theorem}
\begin{remark}
Due to the method which relies on semigroup theory, we need to take measurements during an infinite time. 
\end{remark}
\begin{proof}[Proof of Theorem~\ref{ok}]
For $j=1,2$, let $(v_j,\zeta_j)$ be the solution of the stationary problem~\ref{eq9} with $q=q_j$. According to Proposition~\ref{regstat},  $(v_j,\zeta_j)$ belongs to  $H^4(\Omega)^2 \times H^3(\Omega)$  and moreover, thanks to assumptions~\ref{eq13} and~\ref{eq14},  there exists a constant $C(\alpha,M_1,M_2)>0$ such that 
\begin{equation} \label{majoration_stat}
\|v_j\|_{H^4(\Omega)^2} + \|\zeta_j\|_{H^3(\Omega) } \leq C(\alpha,M_1,M_2).
\end{equation}
We denote $(w_j,\pi_j)=(u_j-v_j,p_j-\zeta_j)$.
Thanks to Theorem~\ref{bebel}, we are able to estimate $\|q_1-q_2\|_{L^2(K)}$ with respect to an increasing function of $(v_1-v_2)_{|\Gamma_e}$, $\displaystyle (\zeta_1-\zeta_2)_{|\Gamma_e}$  and $\left( \frac{\partial \zeta_1}{\partial n} -\frac{\partial \zeta_2}{\partial n}\right)_{|\Gamma_e}$. Our objective is now to compare the asymptotic behavior of $u_1-u_2$ and $p_1-p_2$  to the solution of the stationary problem $v_1-v_2$ and $\zeta_1-\zeta_2$. More precisely, we are going to prove that:
\begin{equation*}
\| w_j(t,.) \|_{H^3(\Omega)^2} + \| \pi_j(t,.)\|_{H^2(\Omega)} \leq G(t),
\end{equation*}
where $G$ is a function which tends to $0$ when $t$ goes to $+ \infty$. This inequality, combined with Theorem~\ref{bebel}, will allow us to conclude the proof of Theorem~\ref{ok}.
\par
We have that $(w_j,\pi_j)$ is the solution of the following problem: for $t>0$,
\begin{equation*} 
      \left\{
      \begin{array}{ccll}
        \displaystyle     \partial_t w-\Delta w  + \nabla \pi 		&	=	&	0,							& \textrm{ in }  \Omega,\\
        \displaystyle                                  div \textrm{ } w 		&	=	&	0,							& \textrm{ in }   \Omega, \\  
	\displaystyle 		\frac{\partial w}{\partial n}  - \pi n&	=	&	0,									& \textrm{ on }   \Gamma_e, \\
	\displaystyle 	\frac{\partial w}{\partial n}  - \pi n+q_jw  		&	=	&	0,									& \textrm{ on }  \Gamma_0,
     \end{array}
     \right.
\end{equation*}
completed with the initial condition $w(0)=u_0-v_j$.
Let $t >0$.
We have from the theory of analytic semigroup that:
\begin{equation} \label{eq20}
w_j(t,.) = e^{-t A_{q_j}} w_j(0,.). 
\end{equation}
Let $\eta>0$.  There exists a constant $C>0$ independent of $q_j$ such that:  
\begin{equation} \label{eq21}
\|A_{q_j}^{\eta} e^{-tA_{q_j}}\| \leq C \frac{e^{-\mu t}}{ t^{\eta}}\textrm{, } t>0 \textrm{, } \eta > 0,
\end{equation}
where $\mu$ is given by~\ref{eq18} and where $\|$ $\|$ is the norm operator. 
Using regularity result for the stationary problem given in Proposition~\ref{regstat}, we have that: 
\begin{equation*}
\|w_j(t,.)\|_{H^3(\Omega)^2} +\|\pi_j(t,.)\|_{H^2(\Omega)}  \leq C(\alpha,M_2)\|\partial_t w_j(t,.) \|_{H^1(\Omega)^2} .
\end{equation*}
Note that, thanks to Proposition~\ref{regularite} we have:
\begin{equation*}
\| \partial_t w_j(t,.) \|_{H^1(\Omega)^2} = \|A_{q_j} w_j(t,.)\|_{H^1(\Omega)^2}. 
\end{equation*}
Then, since $w_j(t,.)$ is given by~\ref{eq20}, and using Proposition~\ref{isometry} and estimates~\ref{majoration_stat} and~\ref{eq21} with $\eta=\frac{3}{2}$, it follows: 
\begin{equation} \label{eq22}
\begin{split}
\|w_j(t,.)\|_{H^3(\Omega)^2}+\|\pi_j(t,.)\|_{H^2(\Omega)}&  \leq C(\alpha,M_2) \|A_q^{\frac{3}{2}} e^{-t A_{q_j}}w_j(0,.)\|_{L^2(\Omega)^2}	 \\
 & \leq \displaystyle C(\alpha,M_2) \frac{e^{-\mu t}}{ t^{\frac{3}{2}}} \left(\|u_0\|_{L^2(\Omega)^2}+ \|v_j\|_{L^2(\Omega)^2}\right) \\
						&   \leq \displaystyle  C(\alpha,u_0,M_1,M_2) \frac{e^{-\mu t}}{ t^{\frac{3}{2}}} .
\end{split}														
\end{equation}
We have from~\ref{eq22}:
\begin{equation*}
\|v_1-v_2\|_{L^2(\Gamma_e)^2} \leq C(\alpha,u_0,M_1,M_2) \frac{e^{-\mu t}}{ t^{\frac{3}{2}}} + \|u_1-u_2\|_{L^{\infty}(0,+\infty;L^2(\Gamma_e)^2)}.
\end{equation*}
Then, passing to the limit when $t$ goes to infinity, we get:
\begin{equation*}
\|v_1-v_2\|_{L^2(\Gamma_e)^2} \leq  \|u_1-u_2\|_{L^{\infty}(0,+\infty;L^2(\Gamma_e)^2)}.
\end{equation*}
We prove similarly:
\begin{equation*}
\|\zeta_1-\zeta_2\|_{L^2(\Gamma_e)}  \leq \|p_1-p_2\|_{L^{\infty}(0,+\infty;L^2(\Gamma_e))},
\end{equation*}
and
\begin{equation*}
\left \| \frac{\partial \zeta_1}{\partial n} - \frac{\partial \zeta_2}{\partial n} \right \|_{L^2(\Gamma_e)} \leq \left \| \frac{\partial p_1}{\partial n} - \frac{\partial p_2}{\partial n} \right \|_{L^{\infty}(0,+\infty;L^2(\Gamma_e))}.
\end{equation*}
To summarize, we have obtained:
\begin{multline*}
\| v_1-v_2\|_{L^2(\Gamma_e)^2} + \|\zeta_1-\zeta_2\|_{L^2(\Gamma_e)} + \left \| \frac{\partial \zeta_1}{\partial n} - \frac{\partial \zeta_2}{\partial n} \right \|_{L^2(\Gamma_e)} \\
\leq \| u_1-u_2\|_{L^{\infty}(0,+\infty;L^2(\Gamma_e)^2)}+ \|p_1-p_2\|_{L^{\infty}(0,+\infty;L^2(\Gamma_e))} +  \left\| \frac{\partial p_1}{\partial n} - \frac{\partial p_2}{\partial n}\right\|_{L^{\infty}(0,+\infty;L^2(\Gamma_e))}.
\end{multline*} 
Applying Theorem~\ref{bebel} to $(v_j,\zeta_j)$ for $j=1,2$, we obtain the existence of positive constants $C(M_1,M_2,\alpha)$  and  $C_1(M_1,M_2,\alpha)$ such that
\begin{equation*}
\|q_1-q_2\|_{L^2(K)} \leq 
\frac{1}{m}\frac{C(M_1,M_2,\alpha)} { \left(\ln \left( \frac{C_1(M_1,M_2,\alpha)}{\| v_1-v_2\|_{L^2(\Gamma_e)^2} + \|\zeta_1-\zeta_2\|_{L^2(\Gamma_e)} + \left \| \frac{\partial \zeta_1}{\partial n} - \frac{\partial \zeta_2}{\partial n} \right \|_{L^2(\Gamma_e)}  }   \right)\right)^{\frac{1}{2} }}.
\end{equation*}
We conclude by using the fact that the function $\displaystyle  x \to \frac{1}{\ln\left(\frac{1}{x}\right)}$ increases on $\mathbb{R}^*_+$.
\end{proof}
\begin{remark}
Remark that 
\begin{equation} \label{plkhtvs}
(u_j,p_j) \in L^{\infty}(0,+\infty;H^3(\Omega)^2) \times L^{\infty}(0,+\infty;H^2(\Omega)).
\end{equation}
Let us prove~\ref{plkhtvs}. Let $\nu >0$. In fact,  thanks to equation~\ref{eq22}, we obtain that 
\begin{equation*}
(w_j,\pi_j) \in L^{\infty}(\nu,+\infty;H^3(\Omega)^2) \times L^{\infty}(\nu,+\infty;H^2(\Omega)),
\end{equation*}
and since $u_j=w_j+v_j$ and $p_j=\pi_j+\zeta_j$, we deduce that 
\begin{equation*}
(u_j,p_j) \in L^{\infty}(\nu,+\infty;H^3(\Omega)^2) \times L^{\infty}(\nu,+\infty;H^2(\Omega)).
\end{equation*}
Moreover, thanks to Corollary~\ref{regularite2}, we have 
\begin{equation*}
(u_j,p_j) \in L^{\infty}(0,\nu;H^3(\Omega)^2) \times L^{\infty}(0,\nu;H^2(\Omega)).
\end{equation*} 
Thus, \ref{plkhtvs} follows.
\end{remark}
%
%
%
\subsubsection{The flux $g$ depends on $t$}
%
%
We restrict our study to the case where $g$ is colinear to the exterior unit normal $n$: $g=\kappa$ $n$.
\\
Let $\alpha>0$, $M_1>0$ and $M_2>0$. We assume that:
\begin{equation} \label{eq90bis}
\kappa \in H^2_{loc}(0,+\infty;H^{\frac{3}{2}}(\Gamma_e)), 
\end{equation}
and 
\begin{equation} \label{eq91}
q \in  H^{\frac{5}{2}}(\Gamma_0) \textrm{ is such that }  \|q\|_{H^{\frac{5}{2}}(\Gamma_0) } \leq M_2 \textrm{ and } q \geq \alpha \text{ on }  \Gamma_0.
\end{equation}
Let us introduce $h$ such that:
\begin{equation} \label{eq90}
h \in H^{\frac{5}{2}}(\Gamma_e) \textrm{ is non identically zero and } \|h\|_{H^{ \frac{5}{2}}(\Gamma_0) } \leq M_1.
\end{equation}
We assume that: 
\begin{equation} \label{hyp_instationnaire}
\begin{split}
\lim_{t \to \infty} & (\|\kappa(t,.)-h\|_{H^{\frac{3}{2} }(\Gamma_e)  }+ \|\partial_t \kappa(t,.)\|_{H^{\frac{3}{2} }(\Gamma_e)  } \\ &+ \left( \int_0^t e^{- \mu (t-s)}  \|\partial_t \kappa(s,.)\|^2_{H^{\frac{3}{2} }(\Gamma_e)  }ds  \right)^{\frac{1}{2}} )  = 0,
\end{split}
\end{equation}
where $\mu$ is given by equation \ref{eq18}.
\begin{theorem} \label{okok}
Let $\alpha >0$, $M_1> 0$, $M_2>0$ and $u_0 \in H^3(\Omega)^2 \cap H$. We assume that $h$ and $\kappa$ satisfy respectively~\ref{eq90} and~\ref{eq90bis} and for $j=1,2$, $q_j$ satisfies~\ref{eq91}. We denote by $(u_j,p_j)$ the solution of system \ref{eq1} associated to $q=q_j$. Let $K$ be a compact subset of $\{ x \in \Gamma_0 / v_1(x) \not= 0\}$, where  $(v_1,\zeta_1)$ is the solution of
\begin{equation*} 
	\left \{   
		\begin{array}{ccll}
                     \displaystyle   - \Delta v + \nabla \zeta			&	=	&	0,							& \textrm{ in }  \Omega, \\ 
                    \displaystyle                       div \textrm{ } v 		&	=	&	0,							& \textrm{ in }  \Omega, \\  
                    \displaystyle             \frac{\partial v }{\partial n} - \zeta n		&	=	&	h n,	 						& \textrm{ on }  \Gamma_e, \\   
                     \displaystyle            \frac{\partial v }{\partial n} - \zeta n	+ q_1v &	=	&	0,							& \textrm{ on } \Gamma_0,  \\
                   \end{array}     
      \right.
 \end{equation*}
and let $m>0$ be a constant such that $|v_1|>m$ on $K$. We assume that~\ref{hyp_instationnaire} is verified.
Then there exist $C(\alpha,M_1,M_2)>0$ and $C_1(\alpha,M_1,M_2)>0$ such that 
\begin{equation*}
\begin{split}
& \|q_1-q_2\|_{L^2(K)} 
\leq \\ &\frac{1}{m} \frac{C(\alpha,M_1,M_2)}  { \left ( \ln \left ( \frac{C_1(\alpha,M_1,M_2)}{  \| u_1-u_2\|_{L^{\infty}(0,+\infty;L^2(\Gamma_e)^2)} + \|p_1-p_2\|_{L^{\infty}(0,+\infty;L^2(\Gamma_e) )} +\left\| \frac{\partial p_1}{\partial n} -\frac{\partial p_2}{\partial n} \right \|_{L^{\infty}(0,+\infty;L^2(\Gamma_e))}  }  \right ) \right ) ^\frac{1}{2}}.
\end{split}
\end{equation*}
\end{theorem}
%
%
\begin{remark}
Let $l \in H^2_{loc}(0,+\infty ; H^{\frac{3}{2}}(\Gamma_e))$  and $h \in H^{\frac{3}{2}}(\Gamma_e)$. Assume that there exists $\theta > 0$ such that: 
\begin{equation*}
\sup_{t \geq 0} e^{t \theta} \left( \| l(t,.)\|_{H^{\frac{3}{2}}(\Gamma_e) } + \| \partial_t l(t,.)\|_{H^{\frac{3}{2}}(\Gamma_e)} \right) < + \infty,
\end{equation*}
Then $\kappa =h+l$ satisfies~\ref{hyp_instationnaire}. We note that a particular case of function satisfying~\ref{hyp_instationnaire} is given by $l(t,x)=\omega(t)\rho(x)$ where $\omega \in H^2_{loc}(0,+\infty)$  , $\rho \in H^{\frac{3}{2}}(\Gamma_e) $ and $\lim_{t \to \infty} e^{t \theta} \omega(t)=\lim_{t \to \infty} e^{t \theta} \omega'(t)=0 $.
\end{remark}
\begin{proof}[Proof of Theorem~\ref{okok}]
For $j=1,2$, we decompose $u_j$ into $u_j=v_j+w_j$ where $(v_j,\zeta_j) \in H^4(\Omega)^2 \times H^3(\Omega)$ is the solution of the stationary problem: 
\begin{equation*} 
	\left \{   
		\begin{array}{ccll}
       \displaystyle                 - \Delta v + \nabla \zeta			&	=	&	0,							& \textrm{ in }  \Omega, \\ 
       \displaystyle                                    div \textrm{ } v 		&	=	&	0,							& \textrm{ in }  \Omega, \\  
        \displaystyle                         \frac{\partial v }{\partial n} - \zeta n		&	=	&	h n,							& \textrm{ on }  \Gamma_e, \\   
        \displaystyle                         \frac{\partial v }{\partial n} - \zeta n	+ q_jv &	=	&	0,							& \textrm{ on } \Gamma_0, \\
                   \end{array}     
      \right.
 \end{equation*}
and $(w_j, \pi_j)$ is solution of the following problem: 
\begin{equation*}
      \left\{
      \begin{array}{ccll}
       \displaystyle      \partial_t w-\Delta w  + \nabla \pi 		&	=	&	0,							& \textrm{ in }(0,+\infty)\times  \Omega, \\
       \displaystyle                                   div \textrm{ } w 		&	=	&	0,							& \textrm{ in }   (0,+\infty)\times\Omega,\\  
	\displaystyle 		\frac{\partial w}{\partial n}  - \pi n		&	=	&	(\kappa-h)n,					& \textrm{ on }  (0,+\infty)\times\Gamma_e,\\
	\displaystyle 	\frac{\partial w}{\partial n}  - \pi n+q_j w  		&	=	&	0,							& \textrm{ on } (0,+\infty)\times\Gamma_0,\\
	\displaystyle     w(0,x)                       				&	=	& u_0(x)-v_j(x), 					& \textrm{ in } \Omega.
     \end{array}
     \right.
\end{equation*}
We would like to perform the same reasoning as in Theorem~\ref{ok}.  More precisely, we are going to prove that:
\begin{equation*}
\| w_j(t,.)\|_{H^3(\Omega)^2} + \|\pi_j(t,.)\|_{H^2(\Omega)} \leq G(t),
\end{equation*}
where $G$ is a function which tends to $0$ when $t$ goes to $+ \infty$.  Since the function $\kappa$ depends on $t$, there will be one more step than in Theorem~\ref{ok} and that is why we assume~\ref{hyp_instationnaire}.  
\par
We divide $(w_j,\pi_j)$ into two terms:  $w_j= u^0_j +\tilde{w}_j$ and $\pi_j=p_j^0+\tilde{\pi}_j$, where $(u^0_j,p^0_j)$ is solution of
\begin{equation*}
      \left\{
      \begin{array}{ccll}
       \displaystyle      \partial_t u^0-\Delta u^0  + \nabla p^0 		&	=	&	0,							& \textrm{ in }  (0,+\infty)\times\Omega,  \\
      \displaystyle                                     div \textrm{ } u^0 			&	=	&	0,							& \textrm{ in }   (0,+\infty)\times\Omega, \\  
	\displaystyle 		\frac{\partial u^0 }{\partial n}- p^0 n			&	=	&	(\kappa-h)n,					& \textrm{ on }   (0,+\infty)\times\Gamma_e, \\
	\displaystyle 	\frac{\partial u^0 }{\partial n}- p^0 n+q_ju^0  		&	=	&	0,							& \textrm{ on }  (0,+\infty)\times\Gamma_0,  \\
	\displaystyle     u^0(0,x)                       					&	=	& 0,								& \textrm{ in } \Omega,
	     \end{array}
     \right.
\end{equation*}
and $(\tilde{w}_j,\tilde{\pi}_j)$ is solution of 
\begin{equation*}
      \left\{
      \begin{array}{ccll}
      \displaystyle       \partial_t \tilde{w}-\Delta \tilde{w}  + \nabla \tilde{\pi} 		&	=	&	0,							& \textrm{ in } (0,+\infty)\times\Omega, \\
      \displaystyle                                   div \textrm{ } \tilde{w} 				&	=	&	0,							& \textrm{ in }(0,+\infty)\times \Omega, \\  
	\displaystyle 		\frac{\partial \tilde{w} }{\partial n} - \tilde{\pi} n			&	=	&	0,							& \textrm{ on  }  (0,+\infty)\times \Gamma_e, \\
	\displaystyle 	\frac{\partial \tilde{w} }{\partial n}- \tilde{\pi} n+q_j\tilde{w}  		&	=	&	0,							& \textrm{ on }(0,+\infty)\times  \Gamma_0, \\
	\displaystyle     \tilde{w}(0,x)                       						&	=	& u_0(x) - v_j(x),			& \textrm{ in } \Omega.
	     \end{array}
     \right.
\end{equation*}
Let $t>0$. Using the same arguments as in the previous subsection, we prove that there exists $C(\alpha,u_0,M_1,M_2)>0$ such that:
\begin{equation} \label{sale1}
\|\tilde{w}_j(t,.)\|_{H^3(\Omega)^2}+\|\tilde{\pi}_j(t,.)\|_{H^2(\Omega)}   \leq C(\alpha,u_0,M_1,M_2) \frac{e^{-\mu t}}{ t^{\frac{3}{2}} }.
\end{equation}

It remains for us to bound $\|u^0_j(t,.)\|_{H^3(\Omega)^2}$ and $\| p^0_j(t,.)\|_{H^2(\Omega)}$. We are going to prove that there exists a constant $C(\alpha,M_2)>0$ such that: 
\begin{equation} \label{tata1}
\begin{split}
& \| u^0_j(t,.)\|_{H^3(\Omega)^2} + \|p^0_j(t,.)\|_{H^2(\Omega)}  \\
 \leq  & C(\alpha,M_2) \left( \int_0^t e^{- \mu (t-s)}  \|\partial_t \kappa(s,.) \|^2_{H^{\frac{3}{2} }(\Gamma_e)  } ds \right)^{\frac{1}{2}} + C(\alpha,M_2) \|\partial_t\kappa(t,.) \|_{H^{\frac{3}{2} }(\Gamma_e)  } 
\\ &+C(\alpha,M_2)\left( \|\kappa(t,.)-h \|_{H^{\frac{3}{2} }(\Gamma_e)  } + e^{- \mu t } \|\kappa(0,.)-h\|_{H^{\frac{3}{2} }(\Gamma_e)  }  \right). \\
\end{split}
\end{equation}
If inequality~\ref{tata1} is satisfied, we can end the proof of Theorem~\ref{okok}:
\begin{equation*}
\|w_1(t,.)-w_2(t,.)\|_{H^3(\Omega)^2} \leq \| u^0_1(t,.)-u^0_2(t,.)\|_{H^3(\Omega)^2}+ \| \tilde{w}_1(t,.)-\tilde{w}_2(t,.)\|_{H^3(\Omega)^2},
\end{equation*}
\begin{equation*}
\| \pi_1(t,.)- \pi_2(t,.)\|_{H^2(\Omega)} \leq \|  p^0_1(t,.)- p^0_2(t,.)\|_{H^2(\Omega)}+ \| \tilde{p}_1(t,.)- \tilde{p}_2(t,.)\|_{H^2(\Omega)},
\end{equation*}
and in the following two estimates, the right hand side tends to $0$ when $t$ goes to infinity thanks to inequalities~\ref{sale1} and assumption~\ref{hyp_instationnaire}.
\par
We introduce $(y_j,\rho_j)$ the solution of
\begin{equation*}
      \left\{
      \begin{array}{ccll}
    \displaystyle  -\Delta y + \nabla \rho 							&	=	&	0,							& \textrm{ in }   \Omega, \\
    \displaystyle                                      div \textrm{ } y				&	=	&	0,							& \textrm{ in }  \Omega,  \\  
 \displaystyle			 \frac{\partial y}{\partial n} - \rho n				&	=	&	(\kappa-h) n,					& \textrm{ on }  \Gamma_e, \\
 \displaystyle		 \frac{\partial y}{\partial n} - \rho n+q_j y  				&	=	&	0,							& \textrm{ on } \Gamma_0,
	     \end{array}
     \right.
\end{equation*}
for all $t>0$.
We know that $ (y_j(t,.),\rho_j(t,.)) \in H^3(\Omega)^2 \times H^2(\Omega)$ and satisfies, thanks to Proposition~\ref{regstat}: 
\begin{equation} \label{te}
\| y_j(t,.)\|_{H^3(\Omega)^2} +\|\rho_j(t,.)\|_{H^2(\Omega)}\leq C(\alpha,M_2) \|\kappa(t,.)-h\|_{H^{\frac{3}{2}} (\Gamma_e)}.
\end{equation}
Remark that  $y_j(t,.)$ belongs to $\mathcal{D}(A_{q_j}^{\frac{3}{2}})$. Indeed, there exists a unique $\tilde{p}(t,.) \in H^3(\Omega)$ solution of
\begin{equation} \label{merde}
      \left\{
      \begin{array}{ccll}
     \displaystyle        \Delta \tilde{p}  	&	=	&	0,							& \textrm{ in  }  \Omega,  \\
	 \displaystyle		\tilde{p}	&	=	&	 \kappa-h,						& \textrm{ on } \Gamma_e,  \\
	 \displaystyle		\tilde{p} 	&	=	&	0,							& \textrm{ on }  \Gamma_0 , \\
  \end{array}
     \right.
\end{equation}
 for all $t>0$ and there exists a constant $C>0$ such that 
\begin{equation} \label{eq28}
\|\tilde{p}(t,.)\|_{H^3(\Omega)} \leq C \|\kappa(t,.)-h\|_{H^{\frac{3}{2}}(\Gamma_e)}.
\end{equation}
Then $(y_j,\rho_j+\tilde{p})$ satisfies
\begin{equation*} 
	\left \{   
		\begin{array}{ccll}
         \displaystyle               - \Delta y_j + \nabla (\rho_j+\tilde{p})			&	=	&	\nabla \tilde{p},			& \textrm{ in } \Omega,\\ 
          \displaystyle                                div \textrm{ } y_j 				&	=	&	0,					& \textrm{ in }  \Omega, \\  
           \displaystyle                       \frac{\partial y_j}{\partial n} -( \rho_j+\tilde{p}) n		&	=	&	0, 					& \textrm{ on }  \Gamma_e, \\   
             \displaystyle                     \frac{\partial y_j}{\partial n} - (\rho_j+\tilde{p}) n	+ q_jy &	=	&	0,					& \textrm{ on }  \Gamma_0,  \\
                   \end{array}     
      \right.
 \end{equation*}
  for all $t>0$.
Remark that, since $\displaystyle \nabla \tilde{p} \in L^2(\Omega)$, we have that $\displaystyle y_j(t) \in \mathcal{D}(A_{q_j})$ by definition of $\mathcal{D}(A_{q_j})$. Notice that the fact that   $g$ is colinear to $n$ is important here to do the \textit{change of variable} in the pressure. We deduce from $\displaystyle A_{q_j} y_j(t)=\nabla \tilde{p}(t) \in V= \mathcal{D}(A_{q_j}^{\frac{1}{2}})$ that $\displaystyle y_j(t) \in \mathcal{D}(A_{q_j}^{\frac{3}{2}})$. Moreover, using Proposition~\ref{isometry} and inequality~\ref{eq28}, there exists a constant $C(M_2)>0$ such that: 
\begin{equation} \label{eq290}
\begin{split}
\| A_{q_j}^{\frac{3}{2}} y_j(t,.)\|_{L^2(\Omega)^2} \leq  &C(M_2) \|A_{q_j} y_j(t,.)\|_{H^1(\Omega)^2} = C(M_2) \|\nabla \tilde{p}(t)\|_{H^1(\Omega)^2}  \\ \leq &C(M_2) \| \kappa(t,.)-h\|_{ H^{\frac{3}{2}}(\Gamma_e)},
\end{split}
\end{equation}
that is to say:
\begin{equation} \label{eq300}
\|  y_j(t,.)\|_{\mathcal{D}(A_{q_j}^{\frac{3}{2}}) } \leq C(\alpha,M_2) \|\kappa(t,.)-h\|_{ H^{\frac{3}{2}}(\Gamma_e)}.
\end{equation}
We can use the same argument, replacing $\kappa-h$ by $\partial_t \kappa$, to prove that $\partial_t y_j(t,.) \in \mathcal{D}(A_{q_j}^{\frac{3}{2}})$ together with the estimate
\begin{equation} \label{eq300ter}
\| \partial_t y_j(t,.)\|_{\mathcal{D}(A_{q_j}^{\frac{3}{2}}) } \leq C(\alpha,M_2)  \| \partial_t \kappa(t,.)\|_{ H^{\frac{3}{2}}(\Gamma_e)}.
\end{equation}
Let us consider $\overline{w}_j=u^0_j - y_j$ and $\overline{p}_j= p^0_j- \rho_j$. The couple $(\overline{w}_j,\overline{p}_j)$ is solution of  
\begin{equation} \label{sale3}
      \left\{
      \begin{array}{ccll}
  \displaystyle          \partial_t \overline{w}-\Delta \overline{w}  + \nabla \overline{p}		&	=	&	-\partial_t y_j,					& \textrm{ in } (0,+\infty)\times \Omega,  \\
  \displaystyle                                       div \textrm{ } \overline{w} 		&	=	&	0,							& \textrm{ in }(0,+\infty)\times \Omega, \\  
\displaystyle			\frac{\partial \overline{w}}{\partial n}  - \overline{p} n			&	=	&	0,							& \textrm{ on }(0,+\infty)\times\Gamma_e,\\
\displaystyle		\frac{\partial \overline{w}}{\partial n}  - \overline{p} n+q_j\overline{w}  		&	=	&	0,							& \textrm{ on }  (0,+\infty)\times\Gamma_0,\\
\displaystyle	    \overline{w}(0,x)                       				&	=	& -y_j(0,x), 						&  \textrm{ in }  \Omega.	
     \end{array} 
     \right.
\end{equation}
We know that $\overline{w}_j$ is given by:
\begin{equation*} 
\overline{w}_j(t,.)= -e^{-t A_{q_j} } y_j(0,.) - \int_0^t e^{-(t -s)A_{q_j} } \partial_t y_j(s,.)ds.
\end{equation*}
Using the family $(\phi^l_{q_j})_{l \geq 1}$ defined by Proposition~\ref{valeur_propre}, we have: $\overline{w}_j(t,.)= \sum_{ l \geq 1}C_l(t)\phi_{q_j}^l$, with
\begin{equation*}
C_l(t)= -e^{- t \lambda_{q_j}^l} (y_j(0,.),\phi_{q_j}^l )_{L^2(\Omega)^2} - \int_0^t e^{-(t-s)\lambda_{q_j}^l} (\partial_t y_j(s,.) , \phi_{q_j}^l)_{L^2(\Omega)^2}ds.
\end{equation*}
Thus, recalling that $(\lambda_{q_j}^l)_{l \geq 1}$ satisfies \ref{eq18} and using Cauchy-Schwarz inequality, there exists $C>0$ such that:
\begin{equation*}
C_l(t)^2 \leq 2 e^{- 2 t \mu}  {(y_j(0,.),\phi_{q_j}^l )}^2_{L^2(\Omega)^2} + C \int_0^t e^{-(t-s)\mu} {(\partial_t y_j(s,.) , \phi_{q_j}^l)}^2 _{L^2(\Omega)^2} ds.
\end{equation*}
We obtain from estimates~\ref{eq300} and~\ref{eq300ter}:
\begin{equation} \label{sale4}
\begin{split}
\|\overline{w}_j(t,.)\|_{\mathcal{D}(A_{q_j}^\frac{3}{2})}  \leq &C(\alpha,M_2)  e^{- \mu t } \|\kappa(0,.)-h\|_{H^{\frac{3}{2}}(\Gamma_e)} \\
&  +C(\alpha,M_2) \left( \int_0^t e^{- \mu (t-s)}  \|\partial_t \kappa(s,.)\|^2_{H^{\frac{3}{2}}(\Gamma_e)  }ds  \right)^{\frac{1}{2}}.
 \end{split}
\end{equation}
Remark that, thanks to Proposition~\ref{regularite} and Proposition~\ref{isometry}, we have:
\begin{equation} \label{sale5}
\begin{split}
\|\overline{w}_j(t,.)\|_{H^3(\Omega)^2} & \leq C(\alpha,M_2) \|A \overline{w}_j(t,.)\|_{H^1(\Omega)^2}  \\ &\leq C(\alpha,M_2) \|A^{\frac{3}{2}} \overline{w}_j(t,.)\|_{L^2(\Omega)^2} \leq C(\alpha,M_2) \|\overline{w}_j(t,.)\|_{\mathcal{D}(A_{q_j}^{\frac{3}{2}})} .
\end{split}
\end{equation}
To summarize, using~\ref{sale5} and~\ref{sale4}, we obtain the estimate:
\begin{equation} \label{toutoutout}
\begin{split}
\|\overline{w}_j(t,.)\|_{H^3(\Omega)^2} 
 \leq &  C(\alpha,M_2) e^{- \mu t } \|\kappa(0,.)-h\|_{H^{\frac{3}{2}}(\Gamma_e)} \\  & + C(\alpha,M_2)\left( \int_0^t e^{- \mu (t-s)}  \|\partial_t \kappa(s,.)\|^2_{H^{\frac{3}{2}}(\Gamma_e)  }ds  \right)^{\frac{1}{2}}.
 \end{split}
\end{equation}
Using now the regularity result for the stationary problem given in Proposition~\ref{regstat}, we have: 
\begin{equation*}
\| \overline{p}_j(t,.)\|_{H^2(\Omega)} \leq C(\alpha,M_2) \left (\|\partial_t y_j(t,.)\|_{H^1(\Omega)^2} + \|\partial_t \overline{w}_j (t,.)\|_{H^1(\Omega)^2} \right).
\end{equation*}
Since $A_{q_j} \overline{w}_j = - \partial_t y_j - \partial_t \overline{w}_j$, we obtain: 
\begin{equation*}
\|\overline{p}_j(t,.)\|_{H^2(\Omega)} \leq C(\alpha,M_2) \left (\|\partial_t y_j(t,.)\|_{H^1(\Omega)^2} + \|A_{q_j} \overline{w}_j (t,.)\|_{H^1(\Omega)^2} \right).
\end{equation*}
Thanks to Proposition~\ref{isometry}, we know that 
\begin{equation*}\|A_{q_j} \overline{w}_j (t,.)\|_{H^1(\Omega)^2} \leq C(\alpha) \|A^{\frac{3}{2}}_{q_j} \overline{w}_j (t,.)\|_{L^2(\Omega)^2}.
\end{equation*} 
Therefore, using~\ref{eq300ter} and~\ref{sale4}, we obtain:
\begin{equation} \label{endendend}
\begin{split}
\| \overline{p}_j(t,.)\|_{H^2(\Omega)} \leq &C(\alpha,M_2) \left(e^{-\mu t}\|\kappa(0,.)-h\|_{H^{\frac{3}{2}}(\Gamma_e)}+\|\partial_t\kappa(t,.)\|_{H^{\frac{3}{2}}(\Gamma_e)} \right) \\ 
&+  C(\alpha,M_2)\left( \int_0^t e^{- \mu (t-s)}  \|\partial_t \kappa(s,.)\|^2_{H^{\frac{3}{2}}(\Gamma_e)  }ds  \right)^{\frac{1}{2}}.
\end{split}
\end{equation}
The estimate~\ref{tata1} follows from $u^0_j= \overline{w}_j + y_j$, $p^0_j= \overline{w}_j+ \rho_j$ and inequalities~\ref{te}, ~\ref{toutoutout} and~\ref{endendend}. 
\end{proof}
%
%
%
\subsection{Conclusion}
\label{conclusion}
To conclude, we have proved, under some regularity assumptions on the open set $\Omega$  and on the solution $(u,p)$ of system~\ref{eq1}, logarithmic stability estimates for the Stokes system with mixed Neumann and Robin boundary conditions. Due to the method which relies on a global Carleman inequality proved in \cite{bukhgeim}, these estimates are valid in dimension 2. 

Our result which, as far as we know, is the first result of this type for Stokes system, could be improved in different ways. A first concern could be to prove a logarithmic stability estimate which is valid in any dimension. This will be the subject of a forthcoming work.
Next, as mentioned in Remark \ref{particular_case}, Robin coefficients are estimated on a compact subset $K \subset \Gamma_0$ which is not a fixed inner portion of $\Gamma_0$ but is unknown and depends on a given reference solution. To obtain an estimate of Robin coefficients on the whole set $\Gamma_0$ or on any compact subset $K\subset \Gamma_0$ is still an open question. 
Finally, in our stability estimates, we need measurements on $\Gamma_e$ of $u$, $p$ and $\displaystyle \frac{\partial p}{\partial n}$, while the identifiability result given by Proposition \ref{identifiabilite} only requires information on  $u$ and $\displaystyle \frac{\partial u}{\partial n}-p n$ on $\Gamma$, where $\Gamma \subseteq \Gamma_e$ is a non-empty open subset of the boundary.
Therefore, it might be interesting to know whether it is possible to obtain a stability inequality with less measurement terms.

%
%
\clearpage
\appendix
\label{section:appendix}
\section{Existence and uniqueness for the unsteady problem }
We study the regularity of the solution of the unsteady problem:
\begin{equation*} 
	\left \{   
		\begin{array}{ccll}
          \displaystyle             \partial_t u - \Delta u + \nabla p			&	=	&	0,							& \textrm{ in } (0,T) \times \Omega, \\ 
           \displaystyle                               div \textrm{ } u 		&	=	&	0,							& \textrm{ in } (0,T) \times \Omega, \\  
          \displaystyle                     \frac{\partial u }{\partial n} - pn		&	=	&	g ,							& \textrm{ on } (0,T) \times \Gamma_e, \\   
          \displaystyle                     \frac{\partial u }{\partial n} - pn	+ qu &	=	&	0,							& \textrm{ on } (0,T) \times \Gamma_0,  \\
          \displaystyle                     u(0,\cdot)					&	=	&	u_0		,				& \textrm{ in } \Omega. 
                   \end{array}     
         \right. 
\end{equation*}
where $q$ only depends on the space variable. We are going to prove Theorem~\ref{reg2}. First of all, as a preliminary result, we prove the following existence result:
\begin{proposition} \label{reg}
Let $T > 0$,  $\alpha >0$ and  $u_0 \in H$. We assume that $g \in L^2(0,T;L^2(\Gamma_e)^d)$ and that $q \in  L^{\infty}(\Gamma_0)$ is such that $q \geq \alpha$ on $ \Gamma_0$. 
There exists $u \in L^2(0,T;V)$ such that for all $v \in V$, we have in the distribution sense on $(0,T)$:
\begin{equation} \label{mem18}
\frac{d}{dt} \int_{\Omega} u \cdot v + \int_{\Omega } \nabla u: \nabla v + \int_{\Gamma_0} q u \cdot v    =  \int_{\Gamma_e}g \cdot v,
\end{equation}
and for all $ v \in V$,
\begin{equation} \label{cond_ini}
\int_{\Omega} u(0) \cdot v = \int_{\Omega} u_0 \cdot v.
\end{equation}
\end{proposition}
\begin{proof}[Proof of Proposition~\ref{reg}]
We begin by proving, using a Galerkin method, that there exists $u \in L^2(0,T;V)$ such that
\begin{equation} \label{eq987}
\begin{split}
&\forall v \in V, \forall \psi \in \mathcal{C}^1(0,T) \textrm{ such that } \psi(T)=0 \\
&- \int_0^T \int_{\Omega} u(t,x)\cdot v(x)  \psi'(t) dxdt + \int_0^T \int_{\Omega} \nabla u(t,x): \nabla v(x) \psi(t) dx dt \\
&+ \int_0^T \int_{\Gamma_0} q(x) u(t,x) \cdot v(x) \psi(t ) dx dt -\psi(0) \int_{\Omega} u_0(x) \cdot v(x) dx \\ &=  \int_0^T \int_{\Gamma_e} g(t,x) \cdot v(x) \psi(t) dx dt .
\end{split}
\end{equation}
Let $(w_i)_{i \in \mathbb{N}^*}$ be a Hilbert basis of $V$ which is also an orthogonal basis of  $H$. For each $n\in \mathbb{N}^*$, we define an approximate solution as follows: we search $u_n \in V_n = \text{Span} \left\{\left(w_i\right)_{1\leq i \leq n}\right\}$ which satisfies 
\begin{equation} \label{eq25}
\left \{
\begin{aligned} 
		& \int_{\Omega}  u_{n,t} \cdot w_j+ \int_{\Omega} \nabla u_n : \nabla w_j + \int_{\Gamma_0} q u_n \cdot w_j = \int_{\Gamma_e}g \cdot w_j , \forall j \in \{1, \ldots, n\}, \\
         		& u_n(0)= \sum_{k=1}^n (u_0,w_k)_{L^2(\Omega)^d}w_k,       
\end{aligned} 
\right.
\end{equation}
where $u_{n,t}$ denotes $\partial_t u_n$.
\\
\\
Let $t  \in (0,T)$. We decompose $u_n(t,.)$ in the Hilbert basis: 
\begin{equation*}
u_n(t,.)= \sum_{i=1}^n \xi_i(t) w_i. 
\end{equation*}
We denote by
\begin{equation*}
A= \displaystyle\left( \int_{\Omega} w_i(x) \cdot w_j(x)dx\right)_{1\leq i,j \leq n}
\end{equation*} 
\begin{equation*}
B=\left( \int_{\Omega} \nabla w_i (x): \nabla w_j(x) + \int_{\Gamma_0}q(x) w_i(x) \cdot w_j(x)dx\right)_{1\leq i,j \leq n}
\end{equation*}
\begin{equation*}
\xi(t)= (\xi_i(t))_{1 \leq i \leq n}
\end{equation*}
and \begin{equation*}
L(t)= \left(\int_{\Gamma_e}g(t,x) \cdot w_i(x)dx\right )_{1 \leq i \leq n}.
\end{equation*}
We can rewrite system~\ref{eq25} in the form: 
\[
\left \{ 
\begin{aligned}
& A \xi'(t) + B \xi(t) = L(t), \\
&\xi(0)=((u_0,w_i)_{L^2(\Omega)^d})_{1\leq i \leq n}.    
\end{aligned}
\right.
\] 
Since the matrix A is invertible, the system has a unique global solution $\xi \in H^1(0,T)^d$.
We are now going to prove that there exists a constant $C>0$ independent of $n \in \mathbb{N}^*$ such that:
\begin{equation}
\sup_{t \in (0,T)} \int_{\Omega}{|u_n|^2} + \int_0^T\int_{\Omega}{|\nabla u_n|^2} +\int_0^T \int_{\Omega} |u_n|^2 \leq C.
\label{eq26}
\end{equation}
Multiplying the first equation of $\ref{eq25}$  by $\xi_j$, summing over $j$ for $j=1,\ldots, n$ and then integrating on $(0,t)$, we obtain:
\begin{equation} \label{mem15}
 \int_0^t \int_{\Omega}{ u_{n,t} \cdot u_n} + \int_0^t \int_{\Omega} |\nabla u_n |^2 +  \int_0^t \int_{\Gamma_0} q |u_n|^2
=  \int_0^t  \int_{\Gamma_e}{g \cdot u_n }
\end{equation}
Let $\epsilon >0$. We have, thanks to Cauchy-Schwartz and Young inequalities: 
\begin{align*}
\int_0^t \int_{\Gamma_e} g \cdot u_n & \leq C \int_0^T\int_{\Gamma_e} |g|^2 + \epsilon \int_0^t \int_{\Gamma_e} |u_n|^2 \leq C \int_0^T\int_{\Gamma_e} |g|^2 + \epsilon \int_0^t  \|u_n\|_{H^1(\Omega)^d}^2.
\end{align*}
Choosing $\epsilon$  small enough and using the fact that  $q \geq \alpha$ on $\Gamma_0$, we obtain:
\begin{equation} \label{la3} 
\sup_{t \in (0,T)} \int_{\Omega} |u_n|^2  + \int_0^T \int_{\Omega} |\nabla u_n|^2 +\int_0^T \int_{\Omega} |u_n|^2 \leq  C \left(\int_0^T\int_{\Gamma_e} |g|^2+ \int_{\Omega}{|u_0|^2}\right).
\end{equation}
This gives $\ref{eq26}$.
%
%
According to inequality~\ref{eq26}, there exists $u \in L^2(0,T;V)$ such that, up to a subsequence, 
\begin{equation*}
u_n \rightharpoonup u \textrm{ in } L^2(0,T;V).
\end{equation*} 
Let $j \in \mathbb{N}^*$. Multiplying the first equation of $\ref{eq25}$ by $\psi \in \mathcal{C}^1(0,T)$ such that $\psi(T)=0$ then integrating on  $(0,T)$, we get, $\forall n \geq j$: 
\begin{equation} \label{meme}
\begin{split}
&\int_0^T \int_{\Omega}{ u_{n,t}(t,x) \cdot w_j(x) \psi(t) dx dt } + \int_0^T \int_{\Gamma_0} q(x) u_n(t,x) \cdot w_j(x) \psi(t)dx dt \\
+& \int_0^T\int_{\Omega} \nabla u_n(t,x): \nabla w_j(x) \psi(t) dxdt = \int_0^T{\int_{\Gamma_e}g(t,x) \cdot w_j(x) \psi(t)dxdt}.
\end{split}
\end{equation}
Taking into account that:
\begin{align*}
  &   \int_0^T \int_{\Omega}{ u_{n,t}(t,x) \cdot w_j(x) \psi(t) dx dt }  \\
= & - \int_0^T \int_{\Omega}{u_n(t,x) \cdot w_j(x) \psi'(t) dx dt }    - \int_{\Omega}{u_n(0,x) \cdot w_j(x) \psi(0) dx}, \\
\end{align*}
we easily pass to the limit when $n$ goes to infinity in~\ref{meme}.
Remark that this inequality is still valid if we replace $w_j$ by any $v \in V$ by continuity. This ends the proof of the existence of $u \in L^2(0,T;V)$ which satisfies~\ref{mem18} in the distribution sense on $(0,T)$. 
\par
Let us finish the proof of Proposition~\ref{reg} by proving that the initial condition~\ref{cond_ini} is satisfied. Let $v \in V$. We deduce from equality $\ref{eq987}$ that $\frac{d}{dt}(u,v)_{L^2(\Omega)^d}  \in L^2(0,T)$. Consequently, the function $t \to (u(t),v)_{L^2(\Omega)^d} $ is continuous. This gives a sense to $(u(0),v)_{L^2(\Omega)^d} $.
Let $\psi \in \mathcal{C}^1(0,T)$ such that $\psi(T)=0$. Multiplying $\ref{mem18}$ by $\psi$ and then integrating on $(0,T)$, we obtain:  
\begin{equation}  \label{end_ini}
 \begin{split}
 -\int_0^T {(u,v)_{L^2(\Omega)^d} \psi'(t)dt}  &+ \int_0^T{a_q(u,v) \psi(t) dt} \\
&=   (u(0,.),v)_{L^2(\Omega)^d} \psi(0) + \int_0^T l(v) \psi(t)dt,
\end{split}
\end{equation}
where we recall that $a_q$ is defined by \ref{eq5} and with $\displaystyle l(v)= \int_{\Gamma_e} g \cdot v$, for $v \in V$. Comparing equality \ref{end_ini} with equality~\ref{eq987}, we obtain $\psi(0) (u(0,.)-u_0,v)_{L^2(\Omega)^d}  = 0$, for all $v \in V$. By choosing $\psi$ such that $\psi(0)\not=0$, equality \ref{cond_ini} follows.
\end{proof}
%
%
%
%
We are now able to prove Theorem~\ref{reg2}.
\begin{proof}[Proof of Theorem~\ref{reg2}]
We will begin by proving that $\partial_t u \in L^2(0,T;H)$, then we  will conclude by using the regularity result for the stationary problem from Proposition~\ref{regstat}. \\
Let $t \in (0,T)$. Multiplying the first equation of $\ref{eq25}$  by $\xi_j'$, summing over $j$ for $j=1,\ldots, n$ and then integrating on $(0,t)$, we obtain:
\begin{equation*}
 \int_0^t \int_{\Omega}| u_{n,t}|^2 +  \int_0^t\int_{\Gamma_0}q u_n  \cdot  u_{n,t} + \int_0^t\int_{\Omega} \nabla u_n: \nabla u_{n,t} =  \int_0^t \int_{\Gamma_e}g \cdot  u_{n,t} .
\end{equation*}
We have:
%
%
\begin{equation*}
\int_0^t \int_{\Gamma_e}g \cdot  u_{n,t} = - \int_0^t \int_{\Gamma_e}  \partial_t  g  \cdot u_n -\int_{\Gamma_e}  g(0)  \cdot u_n(0) +\int_{\Gamma_e}  g(t) \cdot u_n(t).
\end{equation*}
Let $\epsilon >0$. Then, thanks to Cauchy-Schwarz and Young inequalities, there exists $C>0$:
\begin{equation*}
\begin{split}
\left| \int_0^t \int_{\Gamma_e} g  \cdot u_{n,t}\right|    \leq  &\int_0^T \int_{\Gamma_e} |  \partial_t  g |^2 + \epsilon \int_0^T \|u_n\|^2_{H^1(\Omega)^d}  + 2 \sup_{t \in (0,T)} \int_{\Gamma_e} |g|^2  \\  &  + \| u_0\|_{H^1(\Omega)^d}^2 + \epsilon \sup_{t \in (0,T)} \|u_n\|^2_{H^1(\Omega)^d} .
\end{split}
\end{equation*} 
If we choose $\epsilon$ small enough, we finally obtain, using  estimate~\ref{la3}: 
\begin{equation} \label{eq111}
\begin{split}
& \sup_{t \in (0,T)} \|u_n\|^2_{H^1(\Omega)^d} +\int_0^T \|u_n\|^2_{H^1(\Omega)^d}+ \int_0^T \int_{\Omega}| u_{n,t}|^2 \\ & \leq  C\left( \|u_0\|^2_{H^1(\Omega)^d} + \int_0^T \int_{\Gamma_e} |  \partial_t  g |^2 +\int_0^T \int_{\Gamma_e} |  g|^2 +\sup_{t \in (0,T)} \int_{\Gamma_e} |g|^2 \right).
\end{split}
 \end{equation}
We deduce that  $(u_n)_{n \in \mathbb{N}^*}$ is bounded in $H^1(0,T;H)\cap  L^{\infty}(0,T;V)$ and therefore $u \in H^1(0,T;H) \cap  L^{\infty}(0,T;V)$. 
\par
To get regularity in space, we use the regularity result stated in Proposition~\ref{regstat} for the stationary problem. To do so, we notice that, for all $t \in (0,T)$, $(u(t),p(t))$ is solution of system \ref{eq3} 
with $f$ and $g$ replaced by $\partial_t u(t)$ and $g(t)$. So, by Proposition~\ref{regstat} applied with $k=0$, since $(\partial_t u,g)$ belongs to $L^2(0,T;L^2(\Omega)^d)\times  L^2(0,T; H^{\frac{1}{2}}(\Gamma_e)^d)$, we deduce that $(u,p) \in L^2(0,T;H^2(\Omega)^d) \times L^2(0,T;H^1(\Omega))$.
\par 
Let us now prove the uniqueness of solution. Assume that $u_1$ and $u_2$ are two solutions and let $w=u_1-u_2$. Then we have for all $v \in V$: 
\begin{equation} \label{eq29}
\int_{\Omega}\partial_t w (t) \cdot v +\int_{\Omega } \nabla w(t): \nabla v + \int_{\Gamma_0} q w(t) \cdot v    = 0 \textrm{, } w(0)=0.
\end{equation}
Taking $v=w(t)$ in~\ref{eq29}, we find:
\begin{equation*}
\frac{1}{2} \frac{d}{dt} \int_{\Omega} |w(t)|^2 + \int_{\Omega} |\nabla w(t)|^2 +\int_{\Gamma_0} q |w(t)|^2=0,
\end{equation*}
that is to say
\begin{equation*}
\int_{\Omega} |w(t)|^2 \leq \int_{\Omega} |w(0)|^2=0 \textrm{, for all } t \in (0,T).
\end{equation*}
So $u_1=u_2$ on $(0,T) \times \Omega$. To conclude, thanks to system $~\ref{eq1}$, we obtain $p_1=p_2$. 
\end{proof}
\clearpage

\bibliography{biblio}

\begin{thebibliography}{10}

\bibitem{adams_fournier}
R.~A. Adams and J.~J.~F. Fournier.
\newblock {\em Sobolev spaces}, volume 140 of {\em Pure and Applied Mathematics
  (Amsterdam)}.
\newblock Elsevier/Academic Press, Amsterdam, second edition, 2003.

\bibitem{alessandrini_rondi}
G.~Alessandrini, L.~Del~Piero, and L.~Rondi.
\newblock Stable determination of corrosion by a single electrostatic boundary
  measurement.
\newblock {\em Inverse Problems}, 19(4):973--984, 2003.

\bibitem{alessandrini_sincich}
G.~Alessandrini and E.~Sincich.
\newblock Detecting nonlinear corrosion by electrostatic measurements.
\newblock {\em Appl. Anal.}, 85(1-3):107--128, 2006.

\bibitem{Baffico}
L.~Baffico, C.~Grandmont, and B.~Maury.
\newblock Multiscale modeling of the respiratory tract.
\newblock {\em Math. Models Methods Appl. Sci.}, 20, 2010.

\bibitem{ccb}
M.~Bellassoued, J.~Cheng, and M.~Choulli.
\newblock Stability estimate for an inverse boundary coefficient problem in
  thermal imaging.
\newblock {\em J. Math. Anal. Appl.}, 343(1):328--336, 2008.

\bibitem{jbalia}
M.~Bellassoued, M.~Choulli, and A.~Jbalia.
\newblock Stability of the determination of the surface impedance of an
  obstacle from the scattering amplitude.
\newblock {\em Preprint}, http://hal.archives-ouvertes.fr/hal-00659032/fr/,
  2012.

\bibitem{Bergh}
J.~Bergh and J.~L{\"o}fstr{\"o}m.
\newblock {\em Interpolation spaces. {A}n introduction}.
\newblock Springer-Verlag, Berlin, 1976.
\newblock Grundlehren der Mathematischen Wissenschaften, No. 223.

\bibitem{bourgeois}
L.~Bourgeois.
\newblock About stability and regularization of ill-posed elliptic {C}auchy
  problems: the case of {$C^{1,1}$} domains.
\newblock {\em M2AN Math. Model. Numer. Anal.}, 44(4):715--735, 2010.

\bibitem{bourgeois_darde}
L.~Bourgeois and J.~Dard{\'e}.
\newblock About stability and regularization of ill-posed elliptic {C}auchy
  problems: the case of {L}ipschitz domains.
\newblock {\em Appl. Anal.}, 89(11):1745--1768, 2010.

\bibitem{fabrie}
F.~Boyer and P.~Fabrie.
\newblock {\em \'{E}l\'ements d'analyse pour l'\'etude de quelques mod\`eles
  d'\'ecoulements de fluides visqueux incompressibles}, volume~52 of {\em
  Math\'ematiques \& Applications (Berlin) [Mathematics \& Applications]}.
\newblock Springer-Verlag, Berlin, 2006.

\bibitem{brezis}
H.~Brezis.
\newblock {\em Analyse fonctionnelle}.
\newblock Collection Math\'ematiques Appliqu\'ees pour la Ma\^itrise. Masson,
  Paris, 1983.
\newblock Th\'eorie et applications.

\bibitem{bukhgeim}
A.~L. Bukhge{\u\i}m.
\newblock Extension of solutions of elliptic equations from discrete sets.
\newblock {\em J. Inverse Ill-Posed Probl.}, 1(1):17--32, 1993.

\bibitem{chaabane_leblond}
S.~Chaabane, I.~Fellah, M.~Jaoua, and J.~Leblond.
\newblock Logarithmic stability estimates for a {R}obin coefficient in
  two-dimensional {L}aplace inverse problems.
\newblock {\em Inverse Problems}, 20(1):47--59, 2004.

\bibitem{chaabane_jaoua}
S.~Chaabane and M.~Jaoua.
\newblock Identification of {R}obin coefficients by the means of boundary
  measurements.
\newblock {\em Inverse Problems}, 15(6):1425--1438, 1999.

\bibitem{cheng-choulli-lin}
J.~Cheng, M.~Choulli, and J.~Lin.
\newblock Stable determination of a boundary coefficient in an elliptic
  equation.
\newblock {\em Math. Models Methods Appl. Sci.}, 18(1):107--123, 2008.

\bibitem{dautray_lions}
R.~Dautray and J.-L. Lions.
\newblock {\em Analyse math\'ematique et calcul num\'erique pour les sciences
  et les techniques.}
\newblock INSTN: Collection Enseignement. [INSTN: Teaching Collection]. Paris,
  1988.

\bibitem{egloffe_these}
A.-C. Egloffe.
\newblock {\em {\'Etude de quelques probl\`emes inverses pour le syst\`eme de
  Stokes. Application aux poumons.}}
\newblock PhD thesis, Universit\'e Paris VI, 2012.

\bibitem{fabre}
C.~Fabre and G.~Lebeau.
\newblock Prolongement unique des solutions de l'equation de {S}tokes.
\newblock {\em Comm. Partial Differential Equations}, 21(3-4):573--596, 1996.

\bibitem{lunardi_interp}
A.~Lunardi.
\newblock {\em Interpolation theory}.
\newblock Appunti. Scuola Normale Superiore di Pisa (Nuova Serie). [Lecture
  Notes. Scuola Normale Superiore di Pisa (New Series)]. Edizioni della
  Normale, Pisa, second edition, 2009.

\bibitem{pazy}
A.~Pazy.
\newblock {\em Semigroups of linear operators and applications to partial
  differential equations}, volume~44 of {\em Applied Mathematical Sciences}.
\newblock Springer-Verlag, New York, 1983.

\bibitem{phung}
K.-D. Phung.
\newblock Remarques sur l'observabilit\'e pour l'\'equation de laplace.
\newblock {\em ESAIM: Control, Optimisation and Calculus of Variations},
  9:621--635, 2003.

\bibitem{quarteroni}
A.~Quarteroni and A.~Veneziani.
\newblock Analysis of a geometrical multiscale model based on the coupling of
  {ODE}s and {PDE}s for blood flow simulations.
\newblock {\em Multiscale Model. Simul.}, 1(2):173--195 (electronic), 2003.

\bibitem{raviart_thomas}
P.-A. Raviart and J.-M. Thomas.
\newblock {\em Introduction \`a l'analyse num\'erique des \'equations aux
  d\'eriv\'ees partielles}.
\newblock Collection Math\'ematiques Appliqu\'ees pour la Ma\^\i trise.
  [Collection of Applied Mathematics for the Master's Degree]. Masson, Paris,
  1983.

\bibitem{sincich}
E.~Sincich.
\newblock Lipschitz stability for the inverse {R}obin problem.
\newblock {\em Inverse Problems}, 23(3):1311--1326, 2007.

\bibitem{irene}
I.~E. Vignon-Clementel, C.~A. Figueroa, K.~E. Jansen, and C.~A. Taylor.
\newblock Outflow boundary conditions for three-dimensional finite element
  modeling of blood flow and pressure in arteries.
\newblock {\em Comput. Methods Appl. Mech. Engrg.}, 195(29-32):3776--3796,
  2006.

\end{thebibliography}
\bibliographystyle{plain}

\end{document}